\documentclass[letterpaper, 10 pt, conference]{ieeeconf}

\IEEEoverridecommandlockouts

\usepackage{balance}

\usepackage[english]{babel}
\usepackage{amssymb,amsmath} 
\usepackage[latin1,utf8]{inputenc}
\usepackage{graphicx}
\usepackage{epstopdf}
\usepackage{comment}
\usepackage{subcaption}
\usepackage{wrapfig}
\usepackage{cite}
\usepackage[utf8]{inputenc}
\usepackage[mathscr]{euscript}
\usepackage{hyperref}
\usepackage{color}
\usepackage{bm}

\usepackage[english]{babel}
\usepackage[utf8]{inputenc}
\usepackage{graphicx}
\usepackage{listings}
\usepackage{verbatim}
\usepackage{bigints}
\usepackage{comment}
\usepackage{braket}
\usepackage{epigraph}
\usepackage{float}
\usepackage{mathtools}
\usepackage{pgfplots}
\usepackage{algpseudocode}
\usepackage{algorithm}
\usepackage{tikz-network}
\usepackage{tikz}
\usepackage{comment}
\usepackage{hyperref}
\hypersetup{colorlinks, breaklinks, linkcolor=black,citecolor=black,filecolor=black,urlcolor=black} 
\usepackage{graphicx} 
\usepackage{color} 
\usepackage{amsmath,amssymb, amsfonts} 
\usepackage{listings} 
\usepackage{booktabs} 
\usepackage{xspace} 
\usepackage{tabularx}
\usepackage{xcolor}
\usepackage{gensymb}
\usepackage{textcomp}
\usepackage[printonlyused,withpage]{acronym} 
\usepackage{rotating} 
\usepackage{microtype} 
\usepackage{mathrsfs}
\usepackage{cleveref}
\usepackage{tikz-cd}
\usepackage{wrapfig}
\usepackage{dsfont}
\usepackage{lipsum}
\usepackage{longtable}
\usepackage{siunitx}
\usepackage{subfiles}
\graphicspath{{images/}{../images/}}
\usepackage{graphicx}
\usepackage{subcaption}
\usepackage{float}
\usepackage{dsfont}
\usepackage{physics}
\usepackage[normalem]{ulem}
\usepackage{soul}
\usepackage{tikz}
\usetikzlibrary{fit,positioning,arrows,automata}

\DeclareMathOperator{\dom}{dom}

\newcommand{\one}{\mathbf{1}}
\def\mR{{\mathbb R}}

\newcommand{\supp}{\text{supp}}
\DeclareMathOperator{\diag}{diag}

\newcommand{\Tau}{\mathcal{T}}

\def\XXint#1#2#3{{\setbox0=\hbox{$#1{#2#3}{\int}$} 
		\vcenter{\hbox{$#2#3$}}\kern-.5\wd0}}

\renewcommand{\hat}[1]{\widehat{#1}}
\renewcommand{\theta}{\vartheta}
\renewcommand{\epsilon}{\varepsilon}

\def\minwrt[#1]{\underset{#1}{\text{minimize }}}
\def\argminwrt[#1]{\underset{#1}{\text{arg min }}}
\def\maxwrt[#1]{\underset{#1}{\text{maximize }}}
\def\argmaxwrt[#1]{\underset{#1}{\text{arg max }}}
\def\maxemphwrt[#1]{\underset{#1}{\text{\emph{maximize} }}}

\def\minwrt[#1]{\underset{#1}{\text{minimize }}}
\def\argminwrt[#1]{\underset{#1}{\text{arg min }}}
\def\maxwrt[#1]{\underset{#1}{\text{maximize }}}
\def\argmaxwrt[#1]{\underset{#1}{\text{arg max }}}
\def\maxemphwrt[#1]{\underset{#1}{\text{\emph{maximize} }}}
\newcommand{\ett}{{\bf 1}}

\newtheorem{theorem}{Theorem}
\newtheorem{remark}{Remark}
\newtheorem{proposition}{Proposition}
\newtheorem{corollary}{Corollary}
\newtheorem{lemma}{Lemma}
\newtheorem{definition}{Definition}
\newtheorem{example}{Example}

\newtheorem{assumption}{Assumption}

\crefname{assumption}{Assumption}{Assumptions}
\Crefname{assumption}{Assumption}{Assumptions}

\def\ccD{{\mathcal{D}}}

\definecolor{Michele}{HTML}{009B55}

\newcommand{\T}{\Tau}
\DeclareMathOperator{\Span}{span}
\usepackage{subfloat}
\usepackage{stfloats}
\usepackage{xcolor}

\title{\LARGE \bf
A proximal approach to the Schr{\"o}dinger bridge problem with incomplete information and application to contamination tracking in water networks}

\author{Michele Mascherpa, Victor Molnö, Carsten Skovmose Kallesøe and Johan Karlsson
\thanks{This work is supported by KTH Digital Futures, project DEMOCRITUS and the Swedish Research
Council (VR) under grant 2020-03454. The SWIL setup is funded by Poul Due Jensen Foundation.}
\thanks{M.~Mascherpa and J.~Karlsson are with the Division of Optimization and Systems Theory, Department of Mathematics, KTH Royal Institute of Technology, Stockholm, Sweden. {\tt\small micmas@kth.se}, {\tt\small johan.karlsson@math.kth.se}.
V.~Molnö is with the Division of Decision and Control Systems, EECS, KTH Royal Institute of Technology, {\tt\small vmolno@kth.se}.
C.S. ~Kallesøe is with the University of Aalborg, Denmark and with Grundfos, Denmark, {\tt\small csk@es.aau.dk}.}
}

\begin{document}

\maketitle
\thispagestyle{empty}
\pagestyle{empty}

\begin{abstract}
In this work, we study a discrete Schr{\"o}dinger bridge problem with partial marginal observations. A main difficulty compared to the classical Schr{\"o}dinger bridge formulation is that our problem is not strictly convex and standard Sinkhorn-type methods cannot be directly applied. 
To address this issue, we propose a scalable computational method based on an entropic proximal scheme. 
Furthermore, we develop a framework for this problem that includes duality results, characterization of the optimal solutions, and an observability condition that determines when the optimal solution is unique.
We validate the method on the problem of estimating contamination in a water distribution network, where the partial marginals correspond to measured pollutant concentrations at the sensor locations. The experiments were conducted on a laboratory-scale water distribution network.

\end{abstract}

\section{Introduction}

The Schr{\"o}dinger bridge problem is a classical problem in statistical mechanics. Given two observations of a particle distribution at two time instances, and a prior model of the particle evolutions, the Schr{\"o}dinger bridge describes the most likely evolution of the particles between the two distributions. This evolution is characterized as the one that minimizes the relative entropy with respect to the prior while matching the observations.
The Schr{\"o}dinger bridge problem therefore provides a probabilistic framework for describing transport processes evolving over time and has found applications in areas such as stochastic control and inference \cite{opper2019variational, pavon2021data, chen2021stochastic, SinHaaZha20}. Furthermore, its well established connection with optimal transport, in particular with the entropy-regularized formulation  \cite{leonard2013survey, chen2016relation}, allows for the discrete Schr{\"o}dinger bridge problem to be solved efficiently with algorithms based on Sinkhorn iterations \cite{cuturi2013sinkhorn, peyre2019computational, haasler2021pgm}.

A natural extension of the Schr{\"o}dinger bridge problem is the case in which the marginal observations, used to infer the  Schr{\"o}dinger bridge, are not fully observed, and only partial observations are available. This arises, for instance, in networked systems where only a subset of the nodes is observed, but one seeks to infer the evolution of the whole system.  
In this setting, the problem is generally ill-posed, since the total mass of the system is not known a priori. This may lead to a loss of uniqueness of the optimal solution and, moreover, prevents the direct application of classical Sinkhorn-type iterations, which rely on fixed marginal masses.

The Schr{\"o}dinger bridge problem with partial information has been studied in \cite{mascherpa2023estimating} in the context of contaminant spread in water networks, where sensors placed at selected locations provide partial observations of the system at each time step. There, full knowledge of the first marginal (and thus of the contaminant source) is assumed, which fixes the total mass and avoids the aforementioned ill-posedness. The resulting Schr{\"o}dinger bridge problem is addressed via a dual formulation and solved using a dual coordinate ascent scheme.
More recently, \cite{pathan2024entropy} investigated discrete Schr{\"o}dinger bridge problems with incomplete information in the initial--final marginal setting. In that work, neither marginal is assumed to be fully known, but a mass-normalization condition is imposed. This allows the associated Schr{\"o}dinger system to be solved through an iterative procedure similar to Sinkhorn iterations for multimarginal entropic optimal transport.
In a related setting, \cite{chen2025optimal} studies Schr{\"o}dinger bridge problems with unbalanced marginals, in which the initial and final distributions are fully specified but may have different total mass, modeled through diffusive dynamics with killing.

One motivating application for Schr{\"o}dinger bridge problems with partial information arises in contamination tracking in water distribution networks \cite{mascherpa2023estimating}. Such networks are critical infrastructure systems, where accidental or malicious contamination events pose serious risks to public health \cite{art:Islam_EnvironmentalReviews2015,art:Bjelkmar_BMC2017,art:Schijven_WaterResearch2026}. Protection of public health in the presence of contamination threats has therefore been widely studied, including the design of early warning systems and mitigation strategies \cite{art:Rasekh_EnvironmentalModellingSoftware2014,art:preis_WaterResourcesPlanning2006,art:usahandbook2006}. In practice, the effectiveness of such strategies crucially depends on the ability to accurately identify the affected area, for instance to isolate contaminated regions and apply targeted flushing procedures \cite{art:Shafiee_ASCE2017,art:Rathore_ASCE2025}. 
This, in turn, requires accurate information on the spatial and temporal evolution of water quality, typically inferred from limited observations.
The problem of contaminant detection and source identification in water distribution networks has been previously considered using both model-based and data-driven approaches (see, e.g., \cite{art:Laird_ASCE2005,art:Guan_ASCE2006,art:preis_WaterResourcesPlanning2006,art:Costa_WaterResorucesManagement2013,art:Rutkowski_IFAC2018,art:Huang_WaterResourcePlanning2009, art:ELIADESdetection}, and the survey ~\cite{art:ELIADEsurvey}). 
While these approaches focus on application-specific source identification strategies, the present work is concerned with the underlying inference problem arising from partial observations over time, independently of a particular network model or contaminant type.

In this work, we consider a formulation of the discrete Schr{\"o}dinger bridge problem with partial information in which neither a full marginal distribution nor the total mass of the system are assumed to be known. To address the resulting ill-posedness, we propose an entropic proximal point scheme that alternates between Schr{\"o}dinger bridge problems with a fixed first marginal and updates of the unobserved mass components. Each proximal step reduces to a problem that can be solved efficiently using Sinkhorn-type iterations. We further analyze the structure of the solution set and derive conditions under which uniqueness of the optimal solution is recovered, relating them to an observability property of an associated linear dynamical system. The proposed methodology is validated on the contamination tracking in water network application, using data collected at the Smart Water Infrastructures Laboratory (SWIL) \cite{val2021smart}, a modular test facility at Aalborg University, Denmark, designed to emulate water distribution networks.

The main contributions of this paper are as follows:
\begin{enumerate}
    \item We extend the model proposed in \cite{mascherpa2023estimating} to the case of a partially observed initial state and solve the resulting optimization problem with an entropic proximal algorithm. In the pollution-tracking application, this allows us to identify the contamination source, and not only the downstream spread of contaminants, consistently with the information available in realistic scenarios.
    \item We analyze well-posedness under partial observations, a setting in which the Schr{\"o}dinger bridge problem may admit multiple optimal solutions.  We characterize the optimal solution set and derive conditions for uniqueness, in terms of the observability of the underlying time-varying linear system.   
    \item We validate the proposed methodology on experimental contamination data that we collected on a laboratory-scale water distribution network at SWIL.
\end{enumerate}

The paper is structured as follows. In \Cref{section:background}, we set the notation and introduce the discrete Schr{\"o}dinger bridge problem. \Cref{sec:problemform} presents the problem formulation and studies existence and uniqueness of the optimal solution. The computational approach, based on the entropic proximal method, is described in \Cref{sec:computational}. \Cref{sec:waternetwork} introduces the application to contamination tracking in water distribution networks, and \Cref{sec:experiments} presents experimental results using data collected at SWIL. The paper concludes in \Cref{sec:conclusion}.

\section{Background}
\label{section:background}
\subsection{Notation}
By $\odot$, $\oslash$, $\exp(\cdot)$, and $\log(\cdot)$ we denote element-wise multiplication, division, exponential, and logarithm of matrices and vectors. The vector of ones $\one_n \in \mathbb{R}^{n\times 1}$ and the identity matrix $I_n \in \mathbb{R}^{n \times n}$ are used, with the dimension omitted when it is clear from the context. The support $\supp(\cdot)$ of a matrix is the set of its non-zero elements. For any indexed quantity $x$, $x_{[i:j]} := \{x_i,\ldots,x_j\}$. Finally, $\mathbb{R}^n_{\geq 0}$ and $\mathbb{R}^n_{> 0}$ denote the non-negative and strictly positive orthants of $\mathbb{R}^n$.

\subsection{The Schr{\"o}dinger bridge problem}
\label{subsect:schrodinger}
We introduce the discrete Schr{\"o}dinger bridge problem, a maximum entropy problem for discrete time and discrete state spaces originally proposed by Erwin Schr{\"o}dinger in the context of diffusion processes~\cite{schrodinger1931umkehrung, schrodinger1932theorie}. In this framework, particle dynamics are modeled as a discrete Markov chain~\cite{haasler2019estimating, pavon2010discrete}. Specifically, consider an ensemble of indistinguishable particles, evolving over a finite set of $n$ states \( X = \{ X_1, X_2, \dots, X_n \} \). 
Denote by  \(q_t\) the state of a generic particle at time \(t\). Its evolution is governed by a row-stochastic transition matrix \(A_t \in \mathbb{R}_{\geq 0}^{n \times n}\), where $
(A_t)_{ij} = \mathbb{P}(q_{t+1}=X_j \mid q_t=X_i)$.
Given the a priori distribution on the path space induced by these matrices $A_t$, new information may become available in the form of marginal distributions. The goal is then to find a path distribution that satisfies the  fixed marginal constraints while remaining as close as possible to the prior distribution in the sense of Kullback-Leibler divergence, defined as follows.
\begin{definition}
Let \( p \) and \( q \) be two nonnegative vectors or matrices of the same dimension, with $\supp{(p)} \subseteq \supp{(q)}$. The normalized Kullback-Leibler (KL) divergence of \( p \) with respect to \( q \), is defined as
\begin{equation}
\label{eq:kl}
\ccD(p | q) := \sum_i \left( p_i \log\left( \frac{p_i}{q_i} \right) - p_i + q_i \right),
\end{equation}
with the convention that \( 0 \log 0 := 0 \). 
\end{definition}
The KL divergence is an example of $\varphi$-divergence (see, e.g., \cite{teboulle1992entropic}), a class of distance-like functions $d_\varphi$ satisfying, for all $p,q \in \mathbb{R}^n_{>0}$,
\[
d_\varphi(p,q)\ge 0, \; \mbox{ and } \quad d_\varphi(p,q)=0 \iff p=q.
\]
Consider two observed marginal distributions $\mu_0, \mu_1 \in \mR^{n}_{\geq 0}$. The $i$-th element $(\mu_t)_i$ denotes the number of particles in state $X_i$ at time $t$.  Given that the number of particles goes to infinity, the discrete particle distributions can be approximated by densities, allowing to consider particles as continuous quantities. The objective is to find a mass transport matrix $M \in \mathbb{R}^{n \times n}_{\geq 0}$ whose entries $(M)_{ij}$ represent the amount of mass transported from state $X_i$ at time $t$ to state $X_j$ at time $t+1$. To ensure consistency with the observed marginals, $M$ is constrained to satisfy $M \one=\mu_0$ and $M^\top \one=\mu_1$.

The likelihood of observing a transition matrix \( M \) in a large system of particles can be approximated using the KL divergence, as the solution of 
\begin{equation} \label{eq:sb_1T}
\begin{aligned}
	\min_{M \in \mR^{n\times n}_{\geq 0}} \ &  \ccD( M \,|\,\diag(\mu_{0})A) \\
	\text{subject to } \  &  M \one = \mu_{0}  , \ \  M^\top \one = \mu_{1} .
 \end{aligned}
\end{equation}
The right-hand side of the KL divergence in \eqref{eq:sb_1T}, which is the prior on particle evolution, is given by the state transition matrix $A$ rescaled by the initial mass distribution $\mu_0$.

Problem~\eqref{eq:sb_1T} can also be interpreted as an entropy-regularized optimal transport problem. The classical discrete optimal transport problem consists of finding a coupling \( M \in \mathbb{R}^{n \times n}_{\geq0} \) that transports a source distribution \( \mu_0 \) to a target distribution \( \mu_1 \) at minimal total cost, where the cost of moving mass from state \( i \) to state \( j \) is given by a cost matrix \( C \in \mathbb{R}^{n \times n}_{\geq 0} \). This leads to the following linear program:
\begin{equation} \label{eq:basicot}
	\begin{aligned}
		\min_{M \in \mathbb{R}^{n \times n}_{\geq 0 }} \quad & \langle C, M \rangle \\
		\text{subject to} \quad & M \one = \mu_0, \quad M^\top \one = \mu_1,
	\end{aligned}
\end{equation}
where \( \langle C, M \rangle = \sum_{i,j} C_{ij} M_{ij} \) denotes the standard Frobenius inner product. 
The problem can be regularized by adding an entropy term \( \epsilon \ccD(M) \), where \( \epsilon > 0 \) is a small regularization parameter and \( \ccD(M) \) denotes the Kullback--Leibler divergence from \( M \) to the uniform coupling, i.e., \( \ccD(M | \one \one^\top ) \). This modification replaces the linear objective with a strictly convex one and enables efficient solution methods based on dual coordinate ascent, such as the Sinkhorn algorithm~\cite{cuturi2013sinkhorn}. 
By defining the kernel matrix \( K := \exp(-C/\epsilon) \), the entropy-regularized transport problem is equivalent to:
\begin{equation} \label{eq:entropicOT}
\begin{aligned}
	\min_{M \in \mathbb{R}_{\geq 0}^{n \times n}} \ & \ccD(M | K) \\
	\text{subject to} \quad & M \one = \mu_0, \quad M^\top \one = \mu_1.
\end{aligned}
\end{equation}
 For a suitable choice of cost matrix and regularization parameter, entropy-regularized optimal transport and the Schr{\"o}dinger bridge problem coincide, and the latter benefits from the same algorithmic tools and scalability as entropic optimal transport.

The Schr{\"o}dinger bridge problem has also been extended to Markov chains of length~$\T$. 
The most widely studied formulation assumes that only the initial and final marginal distributions, 
$\mu_0$ and $\mu_\T$, at times $0$ and $\T$, are fixed, while the intermediate marginals 
$\mu_t$, for $t=1,\ldots,\T-1$, are treated as unknown variables to be determined as part of the 
optimization problem \cite{pavon2010discrete, haasler2019estimating, haasler2021tree}. 
This formulation then seeks the most likely evolution connecting the prescribed endpoints under a given prior Markov dynamics $A_t$. The bridge connecting the initial and final distributions can then be found as the solution to
\begin{equation} \label{eq:sbp}
\begin{aligned}
	\min_{\substack{M_{[0:\T-1]}, \\ \mu_{[1:\T-1]}}} \ 
	& \sum_{t=0}^{\T} \ccD\!\left( M_t \,\middle|\, \diag(\mu_{t})A_t \right) \\
	\text{subject to } \ 
	& M_t^\top \one = \mu_{t}, \qquad \ \: t = 0,\ldots,\T-1, \\
	& M_{t-1}^\top \one = \mu_{t}, \qquad t = 1,\ldots,\T .
\end{aligned}
\end{equation}

In this paper, we study a variant of \eqref{eq:sbp} in which, instead of fixing the full marginals only at the initial and final times, partial observations of the marginals are available at every time step and only on a subset of the states.

\section{Problem Formulation 
and Properties}
\label{sec:problemform}
We consider a system setup as described in the background section, with $n$ states and transition probabilities given by the matrices $A_t$, $t=0,\ldots,\T-1$.
Let the mass transition matrices be $M_t$, where $(M_t)_{ij}$ denotes the amount of mass moving from state $i$ to state $j$ between time $t$ and $t+1$. 
Further, we assume that partial observation, corresponding to a subset of the states, are available over the $\T$ discrete time steps. These observed states are  indexed by the set \( \pi \subseteq \{1, \ldots, n\} \) with $|\pi|=k\le n$, and the observations are collected into a vector $ \rho_t \in \mathbb{R}^k $,  for $t=0,\ldots,\T$.
To formalize the observation model, we define the matrix $C \in \{0,1\}^{k \times n} $ by $C= \left[ e_{\pi_1},\, e_{\pi_2},\, \ldots,\, e_{\pi_k} \right]^\top$
where $ e_i \in \mathbb{R}^n $ is the $ i $-th standard basis vector. The rows of the matrix $C$ thus extract the observed states.
Complementarily, we define $ \overline{C} \in \{0,1\}^{(n-k) \times n}$ as $
\overline{C} = \left[ e_{j_1},\, e_{j_2},\, \ldots,\, e_{j_{n-k}} \right]^\top $,
where \( \{j_1, \ldots, j_{n-k}\} = \{1, \ldots, n\} \setminus \pi \), so that $\overline{C} $ selects the unobserved states.
Together, $C$ and $ \overline{C}$ partition the state space into observed and unobserved components. In particular $ C^\top C+\overline{C} ^\top \overline{C}=I_n$.

We now formulate the problem of minimizing the KL divergence over time of the mass transportation $M_t$, with respect to the prior $A_t$, while respecting partial measurements and conservation of mass at each time point
\begin{subequations}
\label{eq:primal}
\begin{align}
\min_{M_{[0:\T-1]}} \quad &  \sum_{t=0}^{\T-1} \ccD( M_t \,|\,\diag(M_t \one)A_t)  \\
\text{subject to } \  &  C M_t \one = \rho_{t}, \qquad  \mbox{for } \  t=0,\dots,\T-1, \label{eq:obst}\\
\ &  CM_{\T-1}^\top \one = \rho_\T, \label{eq:finalT}\\
& M_t \one= M_{t-1}^\top  \one,  \ \  \mbox{for } \  t=1,\dots,\T-1. \label{eq:matching}
\end{align}
\end{subequations}
Here, the constraints \eqref{eq:obst} and \eqref{eq:finalT} ensure that the transport matrices $M_t$ are consistent with the partial observations $\rho_t$, whereas \eqref{eq:matching} enforces 
consistency of the mass transitions over time.
\begin{remark}
We emphasize that the key difference between \eqref{eq:primal} and Problem (4) in \cite{mascherpa2023estimating} lies in the constraint $C M_0 \one = \rho_{0}$. The former assumes that at time $t=0$ only partial measurements are available. On the contrary, the latter assumed complete knowledge of the initial state. In relation to the tracking of contaminants in a water networks, this corresponds to the knowledge of the source of pollution, whose identification in this paper becomes part of the problem. 
\end{remark}

We introduce a regularity assumption on \eqref{eq:primal} which ensures existence of primal and dual solutions and rules out instances where the constraints force additional zero entries in the transport, despite the corresponding transition being allowed by the prior.
\begin{assumption}[Regularity]
\label{as:reg}
There exists a feasible solution $M$ to \eqref{eq:primal} such that $(M_t)_{ij}>0$ whenever $(A_t)_{ij}>0$.
\end{assumption}
If some entries of $M$ are forced to be zero by the observations (e.g., when $(\rho_t)_i=0$ implies $(CM_t \one )_{i}=0$), we may fix these entries to zero and work with the equivalent reduced problem on the resulting feasible set. 
The following result holds.
\begin{proposition}
    \label{prop:existence}
    Assume that problem \eqref{eq:primal} is feasible. Then a minimizer exists.
\end{proposition}
\begin{proof}
    See {Appendix~\ref{app:exist}}.
\end{proof}


\subsection{Duality}

Next we derive the corresponding dual problem. The following lemma will be useful.
\begin{lemma} 
\label{lemma:dualconstr}
    Let $\xi \in \mR^n$, and $a \in \mR^n_{\geq 0}$ satisfying $a^\top\one= \one$. Assume also that $\supp(m) \subseteq \supp(a)$. The problem
    \[\inf_{m\in \mR_+^n} \sum_{j=1}^n \Big( m_j \log\frac{m_j}{\bar m a_j} {\Big)}-\langle m,\xi \rangle,
    \]
with  $\bar m=\ett^\top m$, has an optimal solution if and only if $\sum_j a_j\exp(\xi_j)\le 1$. In this case, the minimum value is 0 and the set of optimal solutions is given by
\begin{align*}
&\{0\}  && \mbox{ if }\quad     \sum_{j=1}^n a_j\exp(\xi_j)< 1,\\
&\{m=\alpha a \odot \exp(\xi)\mid {\alpha \geq 0}\} && \mbox{ if }\quad     \sum_{j=1}^n a_j\exp(\xi_j)= 1.
\end{align*}
If $\sum_j  a_j \exp(\xi_j)> 1$, then the objective value tends to $-\infty$ for $m^{(\alpha)}=\alpha a \odot  \exp(\xi)$  as $\alpha \to \infty$.
\end{lemma}

\begin{proof}
Follows from \cite[Lemma~2]{mascherpa2025convex} after absorbing the weights \(a\) into the multipliers, i.e., applying that result to \(\lambda_j:=\xi_j+\log a_j\) on the index set \(\supp(a)\).
\end{proof}
The Lagrangian dual problem can then be formulated as in the following proposition
\begin{theorem}
\label{prop:duality}
Under \Cref{as:reg}, a dual formulation of \eqref{eq:primal} is given by
\begin{equation}
\begin{aligned}
\label{eq:dualnonlinear}
   \max_{\lambda_{[0:\T]}} \quad \; & \ \sum_{t=0}^\T \lambda_t^\top\rho_t   \\
\text{subject to} \ & \diag(u_0)A_0  \diag(u_1)A_1\cdots A_{\T-1}u_{\T}  \leq  \one,
\end{aligned}
\end{equation}
with $u_t:=\exp(C^\top \lambda_t)$, $\lambda_t \in \mR^k$, for $t=0,\ldots, \T$. Moreover, the maximum in \eqref{eq:dualnonlinear} is attained, and for any maximizer $\lambda_{[0:\T]}$ there exist vectors $w_1,\ldots,w_{\T-1}\in\mR^n_{>0}$ with $w_0:=\one$ and $w_\T:=u_\T$ such that every primal optimal solution satisfies, for all $t=0,\ldots,\T-1$,
\begin{equation}
\label{eq:primal_scaling_compact}
M_t^*=\diag(M^*_t\one)\,\diag(u_t\oslash w_t)\,A_t\,\diag(w_{t+1}).
\end{equation}
\end{theorem} \vspace{5pt}
\begin{proof}
Introduce the Lagrange multipliers $\lambda_t$, $t=0,\ldots,\T$, for the observation constraints \eqref{eq:obst}, \eqref{eq:finalT}, and $\nu_t$ for the matching constraints \eqref{eq:matching}. 
The Lagrangian is
\begin{align*}
\mathcal{L}(M,\lambda,\nu)
&\!= \! \sum_{t=0}^{\T-1}\ccD\!\left(M_t\,\middle|\,\diag(M_t\one)A_t\right) \!
+ \! \sum_{t=0}^{\T-1} \! \lambda_t^\top \! (\rho_t \!- \! CM_t\one) \\&
+\lambda_\T^\top(\rho_\T-CM_{\T-1}^\top\one)
+\sum_{t=1}^{\T-1}\nu_t^\top(M_t\one-M_{t-1}^\top\one).
\end{align*}
For each fixed $(t,i)$, the terms in $\mathcal L(M,\lambda,\nu)$ depending on the $i$-th row of $M_t$ are
\(
\ccD\!\left(m\,\middle|\,(\one^\top m)a\right)-\langle m,\xi\rangle,
\)
where $m$ denotes the $i$-th row of $M_t$, $a$ the $i$-th row of $A_t$, and
\[
\xi=(C^\top\lambda_t)_i\one-(\nu_t)_i\one+\nu_{t+1},
\]
with $\nu_0:=0$ and $\nu_\T:=C^\top\lambda_\T$. Hence \Cref{lemma:dualconstr} applies row-wise. Therefore, the dual function is finite if and only if
\[
u_t \oslash w_t\odot(A_t w_{t+1})\le \one,\qquad t=0,\ldots,\T-1,
\]
where $u_t=\exp(C^\top\lambda_t)$ and $w_t=\exp(\nu_t)$, with $w_0:=\one$ and $w_\T:=u_\T$.  
In that case each row subproblem has infimum \(0\), so
\(
\inf_{M\ge0}\mathcal L(M,\lambda,\nu)=\sum_{t=0}^{\T}\lambda_t^\top\rho_t,
\)
and otherwise it equals $-\infty$.
Moreover, Lemma~\ref{lemma:dualconstr} yields that,
for each $(t,i)$, any minimizing $i$-th row is either $0$ (if $\frac{(u_t)_i}{(w_t)_i}(A_t w_{t+1})_i<1$), or it satisfies
\begin{equation}
\label{eq:row-minimizer-dual-form}
(M_t^*)_{ij}
=(M_t^*\one)_i\,\frac{(u_t)_i}{(w_t)_i}(A_t)_{ij}(w_{t+1})_j,
\qquad \forall j,
\end{equation}
when $\frac{(u_t)_i}{(w_t)_i}(A_t w_{t+1})_i=1$.
Collecting \eqref{eq:row-minimizer-dual-form} over $i$ yields \eqref{eq:primal_scaling_compact}, noting that if the minimizing $i$-th row is $0$, then $(M_t^*\one)_i=0$ and the $i$-th row of
\eqref{eq:primal_scaling_compact} is identically zero as well.
Therefore the dual problem can be written as
\begin{align*}
\sup_{\lambda_{[0:\T]},\,\nu_{[1:\T-1]}}  
\ &\sum_{t=0}^\T \lambda_t^\top\rho_t
 \\ 
 \quad \quad  \text{subject to}\quad& 
 u_t\odot(A_t w_{t+1})\le w_t,\ \ t=0,\ldots,\T-1.\ \ 
\end{align*}
Iterating the inequalities gives
\[
\one  \ge \diag(u_0) A_0\diag(u_1)A_1\diag(u_2)\cdots A_{\T-1}u_\T,
\]
and the variables $w_{1},\ldots,w_{\T-1}$  can be removed, yielding \eqref{eq:dualnonlinear}. 
Finally, under Assumption~\ref{as:reg} the supremum is attained (and there is no duality gap) by \cite[Thm.~28.2]{Rockafellar1970}. For a maximizer $\lambda$, one admissible choice of $w$ is obtained by setting $w_\T:=u_\T$ and $w_t:=u_t\odot(A_t w_{t+1})$ for $t=\T-1,\ldots,1$.
\end{proof}

Problem \eqref{eq:primal} may admit multiple solutions. We next establish conditions for uniqueness and we characterize the set of optimizer.
Exploiting the optimality condition \eqref{eq:primal_scaling_compact},  the evolution of the distribution $\mu_{t+1}=(M_t)^\top \one $ can be written as the following linear system
\begin{equation}
\label{eq:systemM}
    \begin{aligned}
        \mu_{t+1}&=\mathcal{A}^\top_t \mu_t ,\\
        \rho_t&=C\mu_t,
    \end{aligned}
\end{equation}
where
\begin{equation}
\label{eq:mathcalA}
    \mathcal{A}_t = \diag\left(u_t \oslash w_t \right) A_t \diag(w_{t+1}),
\end{equation}
for dual optimal $u_t=\exp(C^\top\lambda_t)$ and $w_t=\exp(\nu_t)$. 
The problem of uniquely identifying the initial state $\mu_0$ from the output $\rho_0,\ldots,\rho_\T$ is then the ($\T+1$)-step observability of the discrete time varying linear system \eqref{eq:systemM}.

\subsection{Uniqueness}

A key issue is that, since the optimal transport plan $M^*$ is generally not unique, different optimal solutions may induce different state-transition matrices $\mathcal{A}_t$ and, subsequently, different linear systems of the form~\eqref{eq:systemM}. The next result shows that this ambiguity does not affect the associated unobservable subspace.
\begin{proposition}
\label{prop:uniqueness}
Under \Cref{as:reg}, the system \eqref{eq:systemM} is observable if and only if the system
\begin{equation}
\label{eq:systemA}
\begin{aligned}
\mu_{t+1}&=A_t^\top \mu_t,\\
\rho_t&=C\mu_t,
\end{aligned}
\end{equation}
is observable. Moreover, the unobservable subspaces of the two systems coincide.
\end{proposition}
\begin{proof}
See {Appendix~\ref{app:proof1}}.
\end{proof}
As a consequence, observability can be assessed solely from the prior flows $A_t$ and the matrix $C$, allowing the uniqueness of the primal solution to be determined \emph{a priori} with respect to its computation.
For \eqref{eq:systemA}, observability holds if and only if $\rank(\mathcal O_{\T+1})=n$ \cite[Theorem~4]{weiss1972controllability}, where
\begin{equation}
\label{eq:observabilitymatrix}
    \mathcal O_{\T+1}=
\begin{bmatrix}
C\\
CA_0^\top\\
\vdots\\
CA_{\T-1}^\top\cdots A_0^\top
\end{bmatrix}.
\end{equation}

The kernel $\ker(\mathcal O_{\T+1})$ is the unobservable subspace and describes perturbation directions of the initial state that cannot be detected from $\rho_{[0:\T]}$.

Given \Cref{prop:uniqueness}, starting from any optimal solution we can use $\ker(\mathcal O_{\T+1})$ to characterize the set of optimizers.

\begin{theorem}[Structure of the primal optimal set] 
\label{prop:uniqueness2}
Under \Cref{as:reg}, let $M^*$ be an optimal solution of \eqref{eq:primal}, and let $u_t:=\exp(C^\top \lambda_t)$, $w_t:=\exp(\nu_t)$ be the associated dual scalings, where $\lambda_{[0:\T]}$ are optimal dual variables and $w_{[0:\T]}$ are constructed as in \Cref{prop:duality}. Define $\mathcal A_t$ by \eqref{eq:mathcalA}. Then, any other optimal solution $\tilde M$ of \eqref{eq:primal} can be written as
\begin{equation*}
\tilde M_t = M_t^*+\diag(z_t)\mathcal A_t, \quad z_{t+1}=\mathcal A_t^\top z_t,\quad t=0,\ldots,\T-1,
\end{equation*}
for a sequence $z_0,\ldots,z_\T\in\mR^n$ satisfying $z_0\in\ker(\mathcal O_{\T+1})$ and $M_0^*\one+z_0\ge 0$. 
\end{theorem}

\begin{proof}
Let $\tilde M$ be any other optimal solution and set $\Delta M_t:=\tilde M_t-M_t^*$.  Under \Cref{as:reg}, strong duality holds, hence both $M^*$ and $\tilde M$ minimize the
Lagrangian at the same dual point $(\lambda,\nu)$ (with $u_t=\exp(C^\top\lambda_t)$, $w_t=\exp(\nu_t)$). 
The Lagrangian can be separated over the rows of each $M_t$. Fix $(t,i)$ and let $m^*$ and $\tilde m$ denote the $i$-th rows of $M_t^*$ and $\tilde M_t$. The row subproblem is of the form analyzed in \Cref{lemma:dualconstr}. 
Hence the set of minimizing rows is either $\{0\}$, or
$\{\alpha a:\alpha\ge0\}$, where $a$ is the $i$-th row $a$ of $\mathcal A_t$. Therefore $\tilde m-m^*$ is a scalar multiple of
$a$, and collecting rows yields $\alpha_t\in\mR^n$ such that
\begin{equation}
\label{eq:delta_alpha}
\Delta M_t=\diag(\alpha_t)\,\mathcal A_t, \qquad t=0,\ldots,\T-1.
\end{equation}
Define $z_t:=\Delta M_t\one$ for $t=0,\ldots,\T-1$ and $z_\T:=\Delta M_{\T-1}^\top\one$. Feasibility with respect to the marginal constraints \eqref{eq:obst}--\eqref{eq:matching} imposes
\begin{equation}
\label{eq:delta_constraints}
\begin{aligned}
Cz_t &= 0,                 &\qquad t&=0,\ldots,\T,\\
\Delta M_t^\top\one &= z_{t+1}, &\qquad t&=0,\ldots,\T-1.
\end{aligned}
\end{equation}
By \eqref{eq:delta_alpha} we obtain $z_t=\diag(\alpha_t)\mathcal A_t\one$.
If the minimizing set is $\{0\}$, then the $i$-th row of $\Delta M_t$ is zero, hence $(z_t)_i=(\alpha_t)_i=0$; otherwise \Cref{lemma:dualconstr} implies tightness of the dual constraint and $(\mathcal A_t\one)_i=1$. Thus $z_t=\alpha_t$, so \eqref{eq:delta_alpha} becomes $\Delta M_t=\diag(z_t)\mathcal A_t$. Substituting into \eqref{eq:delta_constraints} yields $z_{t+1}=\mathcal A_t^\top z_t$, proving the recursion.

The condition $Cz_t=0$ and $z_{t+1}=\mathcal A_t^\top z_t$ implies that $z_0$ is in the unobservable subspace of the system $(C,\mathcal A)$. By \Cref{prop:uniqueness}, this is equivalent to $z_0 \in \ker ( \mathcal{O}_{\T+1})$. Finally, $\tilde M\ge 0$ implies $\tilde M_0\one=M_0^*\one+z_0\ge 0$.
\end{proof}
As a consequence, we can use the observability of system \eqref{eq:systemA} to assess uniqueness of the solution to \eqref{eq:primal}.
\begin{corollary}[Uniqueness via observability]
\label{cor:unique}
Under \Cref{as:reg}, if $\ker(\mathcal O_{\T+1})=\{0\}$
(equivalently, \eqref{eq:systemA} is observable), then problem \eqref{eq:primal} has a unique optimizer.
\end{corollary}
\begin{proof}
By \Cref{prop:uniqueness2}, given the optimal solution $M^*$, any other optimizer $\tilde M$ corresponds to some $z_0\in\ker(\mathcal O_{\T+1})$.
If $\ker(\mathcal O_{\T+1})=\{0\}$ then $z_0=0$, hence $z_t=0$ for all $t$ and $\tilde M=M^*$.
\end{proof}


We illustrate non-uniqueness through two minimal examples in which $\ker(\mathcal O_{\T+1})\neq\{0\}$.

\begin{example}[Downstream unobserved component]
\label{ex:nonuniqueness1}
Let $n=2$, $\T=1$, with
\[
C=\begin{bmatrix}1&0\end{bmatrix},\quad
A=\begin{bmatrix}\tfrac12&\tfrac12\\[1pt]0&1\end{bmatrix},
\quad
\rho_0=2,\quad \rho_1=1.
\]
Then every matrix of the form
\[
M^*(\eta)=\begin{bmatrix}1&1\\[1pt]0&\eta\end{bmatrix},\qquad \eta\ge0,
\]
is feasible and has objective value $0$ (since $\ccD(\eta\mid\eta)=0$), hence is optimal.
Equivalently, the solution is non-unique because $\ker(\mathcal O_2)=\Span\{\begin{bmatrix}
    0 & 1
\end{bmatrix}^\top\}$.
\end{example}
In \Cref{ex:nonuniqueness1}, the second state is never observed and lies downstream, with respect to $A$, of the only observed node: once mass enters the second state, it cannot influence the observations. Hence any additional mass on it remains undetected and can be added without changing feasibility or the objective.

Non-uniqueness may also arise from ambiguity upstream of an observed node.

\begin{example}[Indistinguishable upstream sources]
\label{ex:nonuniqueness2}
Let $n=3$, $\T=1$, with
\[
A=\begin{bmatrix}
\tfrac12&0&\tfrac12\\[1pt]
0&\tfrac12&\tfrac12\\[1pt]
0&0&1
\end{bmatrix},
\quad
C=\begin{bmatrix}0&0&1\end{bmatrix},
\quad
\rho_0=0,\quad \rho_1=1.
\]
This models a network in which the first two states send mass into the third one, which is the only observed node. Both initial marginals $\hat\mu_0=\begin{bmatrix}
    2 &0&  0\end{bmatrix}^\top$ and $\hat\mu_0=\begin{bmatrix}
    0 &2&  0\end{bmatrix}^\top$ admit feasible couplings with objective value $0$, by $M^*=\diag(\hat \mu_0)A$. Hence any convex combination of the two corresponding optimal couplings is again optimal, yielding a continuum of optima and non-uniqueness. In this case $\ker(\mathcal O_2)=\Span\{\begin{bmatrix}
        1 & -1 & 0
    \end{bmatrix}^\top \}$.
\end{example}

%
%
\section{Computation of solutions via entropic proximal point method}
\label{sec:computational}
Since the total mass may not be fixed, problem \eqref{eq:primal} cannot be directly solved using Sinkhorn-type methods. Instead, we solve it using an entropic proximal method, which iteratively updates the unknown unobserved component of the initial marginal.
Let $ \eta \in \mathbb{R}^{n-k}_{\ge 0}$ parametrize the mass initially located in the unobserved states, so that $
M_0 \one = C^\top \rho_0 + \overline C^\top \eta$. This reduces \eqref{eq:primal} to the minimization of a convex function in the variable \(\eta\), which we solve with an entropic proximal-point method.

For a closed, convex, proper function \(f\), the entropic proximal-point algorithm consists in the iteration of 
\begin{equation}
\label{eq:proximal}
x^{k} \in \arg\min_{x \ge 0}\Bigl\{ f(x)+\epsilon_k \ccD(x\,|\,x^{k-1})\Bigr\},
\end{equation}
where \(\epsilon_k>0\) and \(\ccD(\cdot\,|\,\cdot)\) denotes the KL divergence \eqref{eq:kl} \cite{teboulle1992entropic}. This is the KL analogue of the classical proximal-point method obtained with the squared Euclidean distance.
To apply \eqref{eq:proximal} to \eqref{eq:primal}, define
\begin{subequations}
\label{eq:minf}
\begin{align}
 f(\eta) = &\inf_{M_{[0:\T-1]}\geq 0}   \ccD( M_0 \,|\,\diag(C^\top \rho_0 +\overline{C}^\top \eta)A_0) \notag \\ 
 & \qquad \qquad +\sum_{t=1}^{\T-1} \ccD( M_t \,|\,\diag(M_t \one)A_t) \label{eq:minf-obj}\\
 &\quad  \text{s.t. } \    C M_t \one = \rho_{t}, \quad \;\; \mbox{for }  t=0,\dots,\T\!-\!1, \label{eq:minf-b}\\
 & \quad\qquad CM_{\T-1}^\top \one = \rho_\T, \label{eq:minf-c}\\
& \quad\qquad M_t \one= M_{t-1}^\top  \one,  \;  \mbox{for }  t=1,\dots,\T\!-\!1, \label{eq:minf-d}\\
& \quad\qquad \overline{C}M_0 \one = \eta. \label{eq:minf-e}
\end{align}
\end{subequations}
Thus \(f(\eta)\) is the optimal value of \eqref{eq:primal} when the unobserved part of the initial marginal is fixed to \(\eta\).

The following result holds.
\begin{lemma}
\label{lem:ccp}
    The function $f$ defined in \eqref{eq:minf} is closed, convex and proper.
\end{lemma}
\begin{proof}
Properness follows from feasibility of \eqref{eq:primal}, as there exists at least one \(\eta\in\mathbb{R}^{n-k}_{\ge 0}\) for which the feasible set of \eqref{eq:minf} is nonempty, and the objective is nonnegative.

To prove convexity, write $f(\eta)=\inf_M F(M,\eta)$,
where $
F(M,\eta):=
\ccD (M_0 | \diag(C^\top\rho_0+\overline C^\top\eta)A_0)
+\sum_{t=1}^{\T-1}\ccD(M_t|\diag(M_t\one)A_t)
+\delta_{\mathcal C}(M,\eta),
$
and \(\mathcal C\) denotes the set of \((M,\eta)\) satisfying the linear constraints \eqref{eq:minf-b}--\eqref{eq:minf-e} and \(M_t\ge 0\).
The function \(F\) is jointly convex, since KL divergence is jointly convex in its two arguments, the dependence on \(\eta\) and \(M_t\one\) is affine, and \(\delta_{\mathcal C}\) is convex. Therefore \(f\) is convex as it is the partial infimum of a jointly convex function. Thus \(f\) is convex, since the partial infimum of a jointly convex function is convex \cite[Section~3.2.5]{boyd2004convex}.
Finally, \(F\) is proper and lower semicontinuous, and for fixed \(\eta\) the constraints determine the total mass in the system, which bounds the feasible set in \(M\). Thus \(F\) is level-bounded in \(M\), locally uniformly in \(\eta\), and \(f\) is lower semicontinuous by \cite[Theorem~1.17]{rockafellar2009variational}. Therefore \(f\) is closed.
\end{proof}

We now apply the entropic proximal-point method to the function \(f\) defined in \cref{eq:minf}.
\begin{proposition}
\label{prop:proximalpoint}
Under Assumption~\ref{as:reg}, the entropic proximal-point algorithm \eqref{eq:proximal}, applied to \(f\), converges to a minimizer of \(f\). Moreover, the minimizer in the proximal step is given by $\eta^*=\overline{C}{M_0^*} \one$ where $M^*_{[0:\T-1]}$ is the solution of\\[-15pt]
\begin{align}
\min_{M_{[0:\T-1]}\geq 0} \  &\; \ccD( M_0 \,|\,\diag(C^\top \rho_0 + \overline{C}^\top \hat{\eta})A_0)  \nonumber\\ 
&\; + \sum_{t=1}^{\T-1} \ccD( M_t \,|\,\diag(M_t \one)A_t) \label{eq:proxstep3}
\\
    \mbox{subject to } &\;\eqref{eq:minf-b}-\eqref{eq:minf-d}.\nonumber
\end{align}

\end{proposition}

\begin{proof}   
By \Cref{lem:ccp}, \(f\) is closed, convex and proper. Since \eqref{eq:primal} admits a minimizer by \Cref{prop:existence}, the minimum of \(f\) is attained. Moreover, under Assumption~\ref{as:reg}, one has \(\dom f\cap\mathbb{R}^{n-k}_{>0}\neq\emptyset\). The convergence result then follows from \cite[Theorem~4.3]{teboulle1997convergence}, by considering a constant regularization parameter $\epsilon_k \equiv 1$ and exact proximal steps, i.e., solving each inner problem to optimality.

Given the current iterate \(\hat\eta\), the proximal step is
\begin{equation*}
\begin{aligned}
 \min_{\eta \geq 0} \inf_{M_{[0:\T-1]}\geq 0} \ &\;  \ccD( M_0 \,|\,\diag(C^\top \rho_0 +\overline{C}^\top \eta)A_0)  \\ 
 &\; +\ccD(\eta | \hat{\eta)}  + \sum_{t=1}^{\T-1} \ccD( M_t \,|\,\diag(M_t \one)A_t) \\
 \mbox{subject to }& \;\eqref{eq:minf-b}-\eqref{eq:minf-e}.
 \end{aligned}
\end{equation*}
Using the constraint \(\overline C M_0\one=\eta\), the first marginal and the proximal term can be combined as
\begin{equation*}
    \begin{split}
       & \ccD\!\left(M_0\,\middle|\,\diag(C^\top\rho_0+\overline C^\top\eta)A_0\right)
+\ccD(\eta\mid\hat\eta)
= \\ &
\ccD\!\left(M_0\,\middle|\,\diag(C^\top\rho_0+\overline C^\top\hat\eta)A_0\right)
    \end{split}
\end{equation*}
since the  $\hat \eta$ terms in the logarithm cancel out.
Next note that the objective function no longer depends on $\eta$, and thus the minimizer is given by $\eta^*=\overline{C}{M_0}^* \one$ where $M^*_{[0:\T-1]}$ is the solution of \eqref{eq:proxstep3}.
\end{proof}

\begin{algorithm}
\caption{Entropic proximal method}\label{alg:entropy}
\begin{algorithmic}
\State Choose $\hat{\eta}>0$
    \While{the change in $\hat{\eta}$ is above tolerance}
        \State $M \gets$ optimal solution of \eqref{eq:proxstep3}
        \State $\hat{\eta} \gets \overline{C} M_0 \one$
    \EndWhile
\end{algorithmic}
\end{algorithm}

The resulting iteration is summarized in \Cref{alg:entropy}, which at each iteration solves the entropy minimization problem \eqref{eq:proxstep3} with fixed initial prior
\(
\hat\mu := C^\top \rho_0 + \overline C^\top \hat\eta.
\)
We now derive an efficient solver for this problem. A dual formulation leads to Sinkhorn-type block coordinate ascent updates involving only matrix-vector products and pointwise operations.

\begin{theorem}
\label{thm:dual-prox}
The dual of \eqref{eq:proxstep3} is
\begin{equation}
\label{eq:dualunconstrained}
\max_{\lambda_{[0:\T]}} \quad
\sum_{t=0}^{\T} \lambda_t^\top \rho_t
-\hat\mu^\top \diag(u_0)A_0\diag(u_1)A_1\cdots A_{\T-1}u_\T,
\end{equation}
where \(\lambda_t\in\mathbb{R}^k\) and \(u_t:=\exp(C^\top\lambda_t)\) for \(t=0,\ldots,\T\).
\end{theorem}

\begin{proof}
Introduce Lagrange multipliers \(\lambda_t\), \(t=0,\ldots,\T\), for the observation constraints \eqref{eq:minf-b}, \eqref{eq:minf-c}, and \(\nu_t\), \(t=1,\ldots,\T-1\), for the matching constraints \eqref{eq:minf-d}. The Lagrangian is
\begin{align*}
\mathcal{L}(M,\lambda,\nu)
\!= \!& \ccD\!\left(M_0\,\middle|\,\diag(\hat\mu)A_0\right) \!
+ \!\sum_{t=1}^{\T-1} \!\ccD\!\left(M_t\,\middle|\,\diag(M_t\one)A_t\right) \\
&+\sum_{t=0}^{\T-1}\lambda_t^\top(\rho_t-CM_t\one)
+\lambda_\T^\top(\rho_\T-CM_{\T-1}^\top\one)\\
&+\sum_{t=1}^{\T-1}\nu_t^\top(M_t\one-M_{t-1}^\top\one).
\end{align*}
For \(t=1,\ldots,\T-1\), minimization over \(M_t\) is as in the proof of \Cref{prop:duality}. Hence, with \(u_t:=\exp(C^\top\lambda_t)\), \(w_t:=\exp(\nu_t)\), and \(w_\T:=u_\T\), it is finite if and only if 
\[
u_t\odot(A_t w_{t+1})\le w_t,\qquad t=1,\ldots,\T-1,
\]
in which case its contribution to the dual function is zero.
For \(t=0\), the terms involving \(M_0\) are
\[
\ccD(M_0\mid \diag(\hat\mu)A_0)
-\langle M_0,\; C^\top\lambda_0\,\one^\top+\one\,\nu_1^\top\rangle.
\]

Since \(\diag(\hat\mu)A_0\) is fixed, the minimization is separable entrywise. Setting the derivative with respect to the entries of $M_0$ to zero gives
\[
M_0^*=\diag(\hat\mu)\diag(u_0)A_0\diag(w_1).
\]
Substituting into the Lagrangian yields the contribution $-(\hat\mu\odot u_0)^\top A_0 w_1.$
Hence the dual problem is
\begin{align*}
\max_{\lambda_{[0:\T]},\,\nu_{[1:\T-1]}}\quad
& \sum_{t=0}^{\T}\lambda_t^\top\rho_t-(\hat\mu\odot u_0)^\top A_0 w_1\\
\text{subject to}\quad
& u_t\odot(A_t w_{t+1})\le w_t,\qquad t=1,\ldots,\T-1.
\end{align*}
Since the objective depends on \(w_t\) only through the term \((\hat\mu\odot u_0)^\top A_0 w_1\), it is maximized by choosing the smallest admissible \(w_1\), that is, by taking equality in the constraints, which gives \(
w_t=u_t\odot A_t w_{t+1},\qquad t=\T-1,\ldots,1.
\)
Substituting recursively in the objective function yields
\[
(\hat\mu\odot u_0)^\top A_0 w_1
=
\hat\mu^\top \diag(u_0)A_0\diag(u_1)\cdots A_{\T-1}u_\T,
\]
which gives \eqref{eq:dualunconstrained}.
\end{proof}

The dual problem \eqref{eq:dualunconstrained} can be solved efficiently by block coordinate ascent. The corresponding updates admit a recursive implementation that is described in the following proposition.

\begin{proposition}
\label{prop:bca}
Define \(\hat{\phi}_0:=\hat\mu\), \(\phi_\T:=\one\), and recursively
\begin{subequations}
    \begin{align}
\hat{\phi}_{t+1} &= A_t^\top(\hat{\phi}_t\odot u_t), \qquad \quad \ \  t=0,\ldots,\T-1, \label{eq:hatphi-rec}\\
\phi_t &= A_t(u_{t+1}\odot\phi_{t+1}), \qquad t=\T-1,\ldots,0. \label{eq:phi-rec}
\end{align}
\end{subequations}
Then block coordinate ascent applied to \eqref{eq:dualunconstrained} updates \(u_t\) according to
\begin{equation}
\label{eq:u-update}
C u_t = \rho_t \oslash C(\phi_t\odot\hat{\phi}_t),
\quad
\overline C\,u_t = \one,
\quad t=0,\ldots,\T.
\end{equation}
These updates are summarized in \Cref{alg:cap2}, and the resulting iterates converge to a maximizer of \eqref{eq:dualunconstrained}.
\end{proposition}
\begin{algorithm}
\caption{Block coordinate ascent for \eqref{eq:dualunconstrained}}
\label{alg:cap2}
\begin{algorithmic}
\State Initialize  \(\hat{\phi}_0 \gets \hat{\mu}\), \(\phi_\T \gets \one\), \(u_t>0\), \(t=0,\ldots,\T\)
\For{\(t=\T-1,\ldots,0\)}
    \State \(\phi_t \gets A_t(u_{t+1}\odot\phi_{t+1})\)
\EndFor
\While{the observation residual is above tolerance}
    \For{\(t=0,\ldots,\T\)}
        \State \(u_t \gets C^\top\! (\rho_t \oslash C(\phi_t\odot\hat{\phi}_t))+\overline C^\top\overline C\one\)
         \State  $\hat{\phi}_{t+1} \gets (A^{t})^\top (\hat{\phi}_{t} \odot u_{t})$ \textbf{ if } $t< \T$
    \EndFor
    \For{\(t=\T-1,\ldots,0\)}
        \State \(\phi_t \gets A_t(u_{t+1}\odot\phi_{t+1})\)
    \EndFor
\EndWhile
\end{algorithmic}
\end{algorithm}

\begin{proof}
We solve \eqref{eq:dualunconstrained} by block coordinate ascent, that is, by maximizing the dual objective with respect to one block \(\lambda_t\) at a time while keeping the remaining variables fixed. 
By the recursions \eqref{eq:hatphi-rec}--\eqref{eq:phi-rec}, the vector \(\hat{\phi}_t\) collects the factors to the left of \(u_t\), while \(\phi_t\) collects those to the right. Hence
\[
\hat\mu^\top \diag(u_0)A_0\diag(u_1)\cdots A_{\T-1}u_\T
=
(\phi_t\odot\hat{\phi}_t)^\top u_t,
\quad \forall t.
\]
Taking the derivative of the objective function with respect to $\lambda_t$ and setting it to zero gives the optimality condition
\[
\rho_t = C(\phi_t\odot u_t\odot\hat{\phi}_t), \qquad t=0,\ldots,\T.
\]
Due to the structure of \(C\), this is equivalent to
\[
C u_t = \rho_t \oslash C(\phi_t\odot\hat{\phi}_t),
\qquad
\overline C\,u_t=\one,
\]
that is, $u_t$ is updated in the observed coordinate and set to one in the unobserved ones.
These are precisely the updates implemented in \Cref{alg:cap2}. Since the dual objective is continuously differentiable and concave, convergence of block coordinate ascent follows from \cite[Proposition~2.7.1]{bertsekas1999nonlinear}.
\end{proof}

\begin{remark}
Note the difference between the optimization problems \eqref{eq:proxstep3} and problem (4) in \cite{mascherpa2023estimating} in how the first time step is handled.
Although the dual objective \eqref{eq:dualunconstrained} has a similar form to the corresponding dual in \cite{mascherpa2023estimating}, the factor at \(t=0\) is now determined by partial observations and the current proximal iterate, rather than by a fully prescribed initial marginal.

\end{remark}

\begin{remark}
As stated, \Cref{alg:entropy} requires solving \eqref{eq:proxstep3} at each outer iteration. In practice, this can be made much cheaper by performing only a few sweeps of the inner block coordinate ascent iterations in \Cref{alg:cap2}, often just one. If the \(k\)-th proximal subproblem in \eqref{eq:proximal} is solved with accuracy \(\sigma_k\) such that
\(
\sum_{k=1}^\infty \epsilon_k \sigma_k < \infty,
\)
then the corresponding inexact proximal iterations still converge \cite[Theorem~4.3]{teboulle1997convergence}. 
\end{remark}

\section{Application to Water Networks}
\label{sec:waternetwork}
In this section, we specialize the proposed framework to pollution transport in water distribution networks. We first recall the probabilistic water-flow model of \cite{mascherpa2023estimating}, which provides the transition matrices used as prior information in the Schr{\"o}dinger bridge formulation. We then show how the resulting model is used to estimate contamination from sensor measurements.
\subsection{Water Networks Modeling}
\label{sec:watermodeling}

Consider a water distribution network composed of $n$ interconnected pipes, each represented as a state in a dynamic system. The system is augmented with an absorbing state to account for water that exits the network. 
Pollution transport in the network is observed across $\mathcal{T}$ discrete time steps. We assume that pollution is homogeneously diluted in water and that, at bifurcations, water splits proportionally to the flow rates among downstream branches. 
The hydraulic properties of the network, such as pipe lengths and diameters, are assumed known, and for each time step, we either measure or estimate the flow rate in each pipe. This is feasible when the network includes flow or pressure sensors, as in the SWIL laboratory. 
In the absence of direct access to real-time flow data, water flows in distribution networks can be estimated using historical consumption patterns, see, e.g., \cite{brentan2018water,alvisi2003stochastic}.
These estimates provide macroscopic information on the flow of water through the network at each time step. 
Under these assumptions, the probabilistic motion of individual particles can be inferred and modeled as a time-inhomogeneous Markov chain over the $n$ states. Each particle moves independently, and transitions between pipes are governed by a sequence of transition probability matrices $A_{[0: \T-1]}$, where each element $(A_t)_{ij}$ encodes the probability that a particle in pipe $i$ at time $t$ will be in pipe $j$ at time $t+1$.
These transition probabilities are determined by the flow dynamics. Let $F_i(t)$ denote the flow rate in pipe $i$ at time $t$, and $V_i$ the volume of pipe $i$. Then the normalized water speed in pipe $i$ is defined as
\begin{equation} \label{eq:speed}
    S_i(t) := \frac{F_i(t) \, \Delta t}{V_i}.
\end{equation}
The quantity $S_i(t)$ represents the proportion of water in relation to the pipe volume that exits (or enters) the pipe
during the time interval $(t, t+\Delta t)$. 
For example:
\begin{itemize}
    \item If $S_i(t) < 1$, only part of the volume is replaced, but some stays in the pipe. The networks used in this paper typically reside in this regime.    \item If $S_i(t) = 1$, the water contained in the pipe $i$ is flushed out exactly in one time step. 
    \item If $S_i(t) > 1$, the pipe is flushed entirely in one time-step, and excess water continues downstream.
\end{itemize}
Given the water speeds along the network, the transition probabilities for each pipe can be derived explicitly. In simple line networks without bifurcations, the fraction of water from a given pipe that reaches downstream pipes can be computed from the following result \cite[Proposition~3]{mascherpa2023estimating}. 
\begin{proposition}
\label{prop:pipes}
    Consider the sequence of pipes $\mathcal{I}=\{1,\dots,n\}$ with corresponding speed of water $\{S_1,\dots,S_n\}$ (in pipe units), and assume that $\sum_{k=2}^n S_k^{-1}>1$. 
The proportion of water from pipe $1$ that moves to pipe $k$ is given by
\begin{align}\label{eq:watertransitions}
	a_{1k}=\begin{cases}
	{(1-\alpha)S_1}/{S_{k}} \quad & \mbox{if } k=n_1<n_2\\
	{S_1}/{S_k} \quad &\mbox{if } n_1 < k < n_2 \\
	{\beta S_1}/{S_{k}} \quad &\mbox{if }  k=n_2, n_1<n_2\\
	1 \quad &\mbox{if } k=n_1=n_2\\
	0\quad &\mbox{otherwise}
	\end{cases}
\end{align} 
where the parameters $n_1, n_2\in \mathbb{N}$ and $\alpha, \beta \in [0,1)$  are uniquely specified by the equations
 \begin{equation*}
      \sum_{k=1}^{n_1-1} \frac{1}{S_k}+\frac{\alpha}{S_{{n_1}}}=1,  \qquad
	 \sum_{k=2}^{n_2-1} \frac{1}{S_k}+\frac{\beta}{S_{{n_2}}}=1. 
 \end{equation*}
 \hfill $\square$
\end{proposition}
In the presence of bifurcations, each possible path from a pipe is treated as a separate subproblem
, and the resulting probabilities are combined as a weighted sum, where the weights correspond to the fraction of flow splitting into each branch. An illustration of this procedure is provided in \Cref{fig:three-plots}, which shows a simple water network composed of four pipes and one absorbing state.

\begin{figure*}[!t]
\label{fig:mc}
    \centering
    \begin{minipage}[t]{0.32\textwidth}
        \centering
        \includegraphics[width=\linewidth]{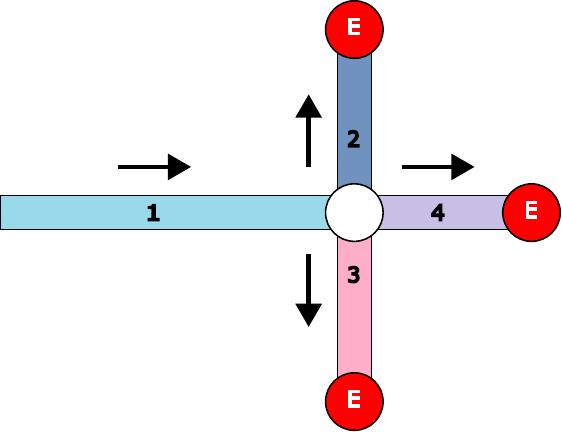}
        \caption*{(a) $t=0$}
    \end{minipage}
    \hfill
    \begin{minipage}[t]{0.32\textwidth}
        \centering
        \includegraphics[width=\linewidth]{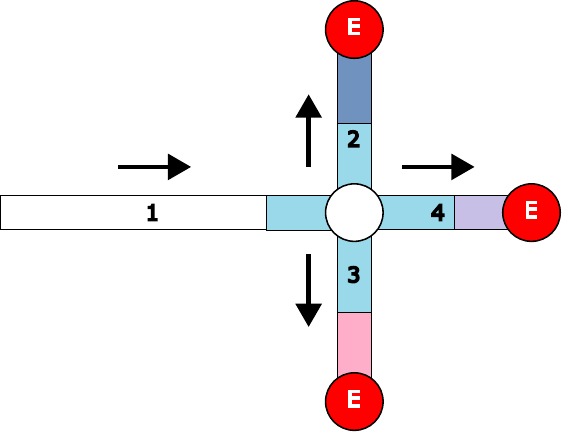}
        \caption*{(b) $t=1$}
    \end{minipage}
    \hfill
    \begin{minipage}[t]{0.32\textwidth}
        \centering
        \includegraphics[width=\linewidth]{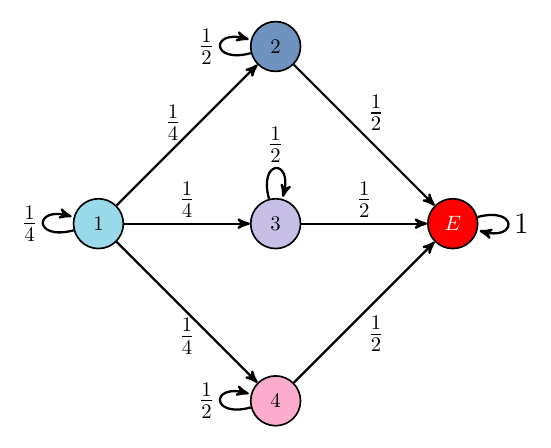}
        \caption*{(c) Resulting MC}
    \end{minipage}
   \caption{Illustration of water network modeling. Each pipe at $t=0$ is marked with a distinct color. At $t=1$, water flows downstream. The red node \(E\) represents the absorbing exit state.
   Transition probabilities are derived by comparing the volume of each color across time steps. Here, \(S_1=3/4\), \(S_2=1/2\), \(S_3=1/2\), and \(S_4=1/2\). To compute the transitions from pipe 1, \Cref{prop:pipes} are applied to the line networks $\{1,2\}$, $\{1,3\}$, $\{1,4\}$, and in each branch it hold that $n_1=1$ and $n_2=2$. We assume that the volumes of the pipes $2, 3, 4$ are the same, thus the proportion of water enter each of them is the same.}
    
    \label{fig:three-plots}
    \vspace{-8pt}
\end{figure*}

\subsection{Estimation of pollution}
\label{sec:modelingpollution}

In the setting of \Cref{sec:watermodeling}, the network consists of $n$ pipes, representing the states, and pollution evolves over $\T$ time steps according to transition matrices $A_t$, $t=0,\ldots,\T-1$, estimated from the water flow. We denote by $M_t$ the mass transition matrix, where $(M_t)_{ij}$ is the amount of pollutant moving from state $i$ to state $j$ between times $t$ and $t+1$.

We assume that $k$ sensors are placed at selected locations in the network, encoded by the binary observation matrix $C \in \{0,1\}^{k \times n} $. These sensors provide measurements of the amount of contaminants at each time $t=0,\ldots,\T$, and the measurements are collected into $\T+1$ vectors $ \rho_t \in \mathbb{R}^k $. With this interpretation, we aim to solve problem \eqref{eq:primal}, minimizing over time the KL divergence of the pollution flow $M_t$, with respect to the prior $A_t$, while respecting sensors measurements and conservation of mass at each time point.

\section{Experiments and results}
\label{sec:experiments}
We present the results of the proposed methodology with data collected at the Smart Water Infrastructure Laboratory (SWIL), at the University of Aalborg. The experiments are designed to assess the ability of the method to reconstruct and localize contaminant transport in water networks under different contamination sources and network configurations.

\subsection{Experimental setup and data}
SWIL is a modular test facility designed to emulate water distribution networks under controlled conditions \cite{val2021smart}. The laboratory setup consists of interconnected pipe modules, pumping stations, and consumer units, allowing flexible reconfiguration of network topology. 
Pictures from the laboratory are presented in \Cref{fig:swil}. 
\begin{figure}
    \centering
    \begin{minipage}[b]{0.49\columnwidth}
        \centering
     \includegraphics[width=\linewidth]
     {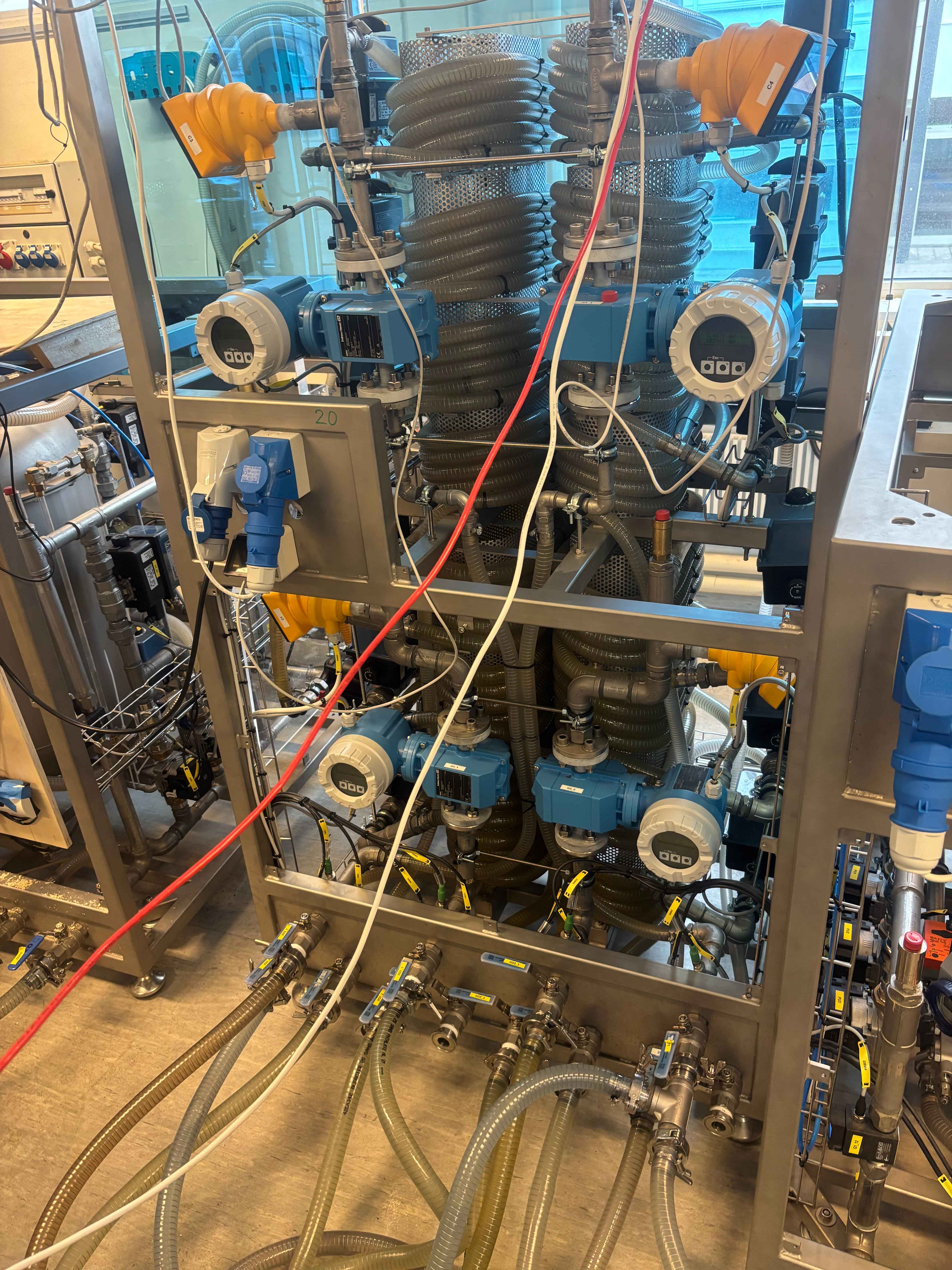}
    \end{minipage}
    \hfill
    \begin{minipage}[b]{0.49\columnwidth}
        \centering
    \includegraphics[width=\linewidth]
    {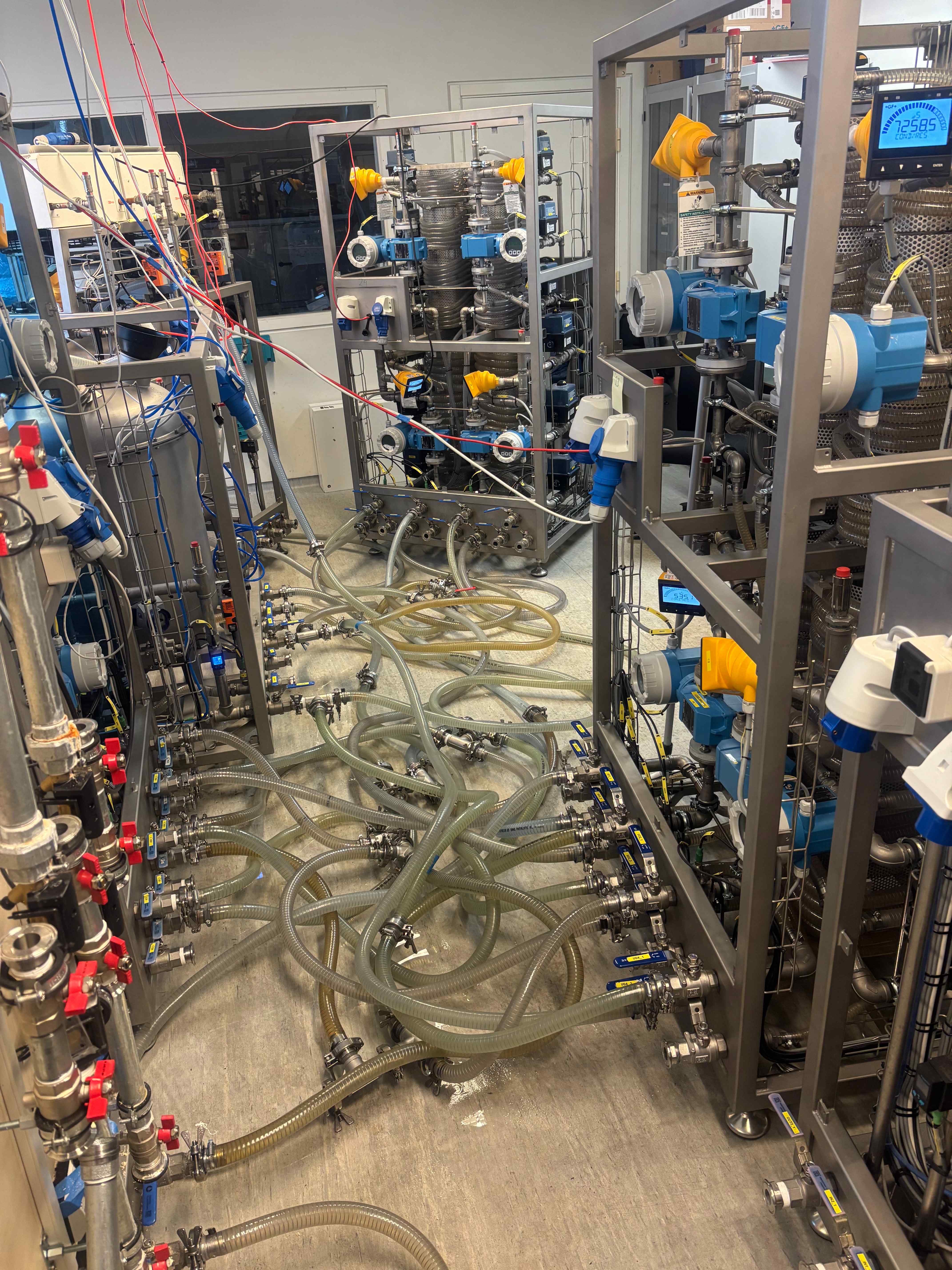}
    \end{minipage}
    \caption{Experimental setup at SWIL. A pipe unite (left), with conductivity sensors (in yellow), used to detect salinity. Several modules are then connected with pipes (right). }
    \label{fig:swil}
    \vspace{-20pt}
\end{figure}
Each pipe is equipped with a flow sensor and a conductivity sensor, providing measurements of flow rate and electrical conductivity at that location. Salt is employed as a tracer, so that conductivity measurements serve as a proxy for contaminant concentration.

For modeling purposes, each physical pipe is discretized into multiple segments, which define the state space of the network model, with the discretization chosen such that each state corresponds to a pipe volume not exceeding 1.5~L. The state space is augmented with additional states representing the pumping stations, as well as an absorbing state accounting for mass exiting the network. Measurements are collected over discrete time steps and preprocessed to obtain estimates of contaminant mass, yielding time series at a resolution of $1$~s.  Flow measurements are used to reconstruct the transition probability matrices $A_t$ governing contaminant transport, while conductivity measurements provide partial observations of the contaminant mass distribution at the corresponding states, as described in \Cref{sec:waternetwork}. 
 Only a subset of the available conductivity sensors is used for estimation, while the remaining ones are used for validation of the reconstructed contaminant distributions.

\subsection{Experimental scenarios}
Two contamination scenarios are considered: contamination occurring in a storage tank, reported as a common and well-documented case in drinking water distribution systems \cite{renwick2019potential}, and contamination occurring in a pipe, which can be observed in connection with deteriorating 
pipe infrastructure \cite{fontanazza2015contaminant}.

In both experiments, contaminant transport is reconstructed using the proposed entropic proximal method. The algorithm is run until the change between successive iterates falls below $10^{-8}$, typically requiring on the order of $10^3$--$10^4$ iterations, using an inexact scheme with two inner sweeps. In both scenarios, non-uniqueness arises only in states downstream of the sensors, and we present here the solution with zero initial mass in those states. 

\paragraph{Contaminated tank}
This setup contains two pumping stations (integrated with a tank) denoted with $P1$ and $P2$, two consumer units $C1$ and $C2$.  Data from two conductivity sensors placed near the consumers are used. The network is illustrated in \Cref{fig:doubletank}, while the properties of the pipes are described in \Cref{tab:contaminatedtank}. After discretization,
the system consists of $n=60$ states, with $k=2$ sensors. The time frame of interest is $\T=196$ seconds. 
\begin{figure}[ht]
    \centering
    \includegraphics[width=\linewidth]{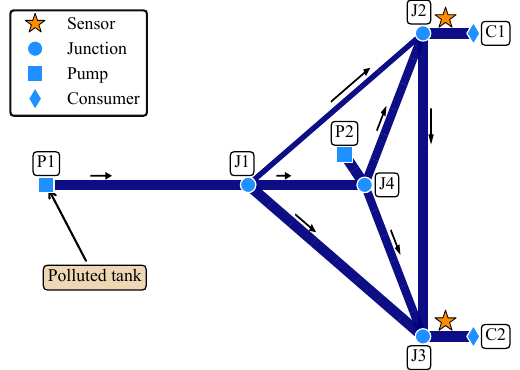}
    \caption{Water network for the \textit{contaminated tank} scenario. Water flows according to the arrows, from the tanks to the consumers. The widths of the segments are scaled to the pipe diameters, but their lengths do not correspond to the actual pipe lengths.}
    \label{fig:doubletank}
\end{figure}
Salty water with $150g$ of salt is introduced in the tank $P1$ and thoroughly mixed by activating the pump, running the water through a separate loop to ensure uniform salinity. Once the experiment begins, fresh water is continuously supplied to the contaminated tank to maintain the system’s operation over a longer period. The system is built in such a way that each consumer receives water from every pumping station.

The estimate is reported in \Cref{fig:simulationtank} for a sequence of time steps. 
Furthermore, to better evaluate the result, we compare it with data from sensors not included in the model, each of them located at the inlet of the corresponding pipe. This comparison, together with the signals from the sensors included in the model, is reported in \Cref{fig:plots_side_by_side}.
The method correctly identifies the contaminated tank as the source of pollution, as can be observed in \Cref{fig:simulationtank} for $t=5$~s, and it estimates at $t=0$, a total of $96g$ between the tank and the inlet of $(P1,J1)$. 
The estimated total mass of salt is consistent with the experimental setup. Out of the initial $150 g$ introduced in the system, roughly $25 g$ are expected to remain in the separate mixing loop used to dissolve the salt.  
Near the end of the experiment, the available time window is insufficient for all contaminated water to reach the sensors, which affects the reconstruction for $t \geq 125$.
This behavior is also visible in \Cref{fig:plots_side_by_side}, where the estimates remain approximately constant while the measured signals diminish.
The method reproduces the qualitative transport pattern of the contaminant, identifying all pipes that are actually crossed by the polluted flow. The only discrepancy is in $(J2,J3)$, in which the sensors reading is zero, possibly due to a non-uniform mixing taking place in the node $J2$.

\begin{figure*}[!t] 
	\centering
	\includegraphics[width=0.97\textwidth]{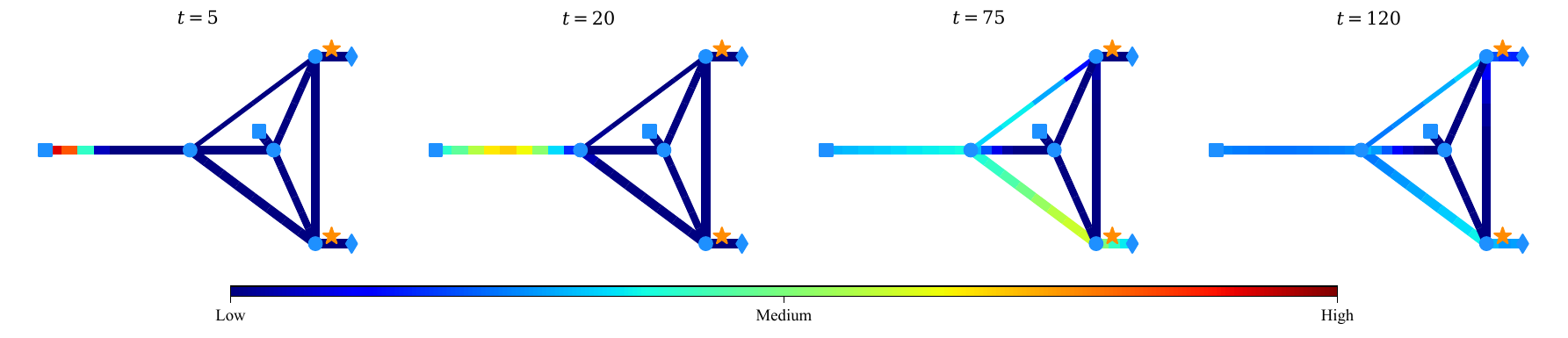}
	\caption{Concentration of salt over time in the reconstructed solution for the \textit{contaminated tank} experiment, obtained using measurements over the full observation horizon ($\T=196$~s). Selected representative time instants are shown.}
	\label{fig:simulationtank}
	\vspace{-8pt}
\end{figure*}

\begin{figure*}[!t]
    \centering
    \begin{minipage}{0.48\textwidth}
        \centering
        \includegraphics[width=\linewidth]{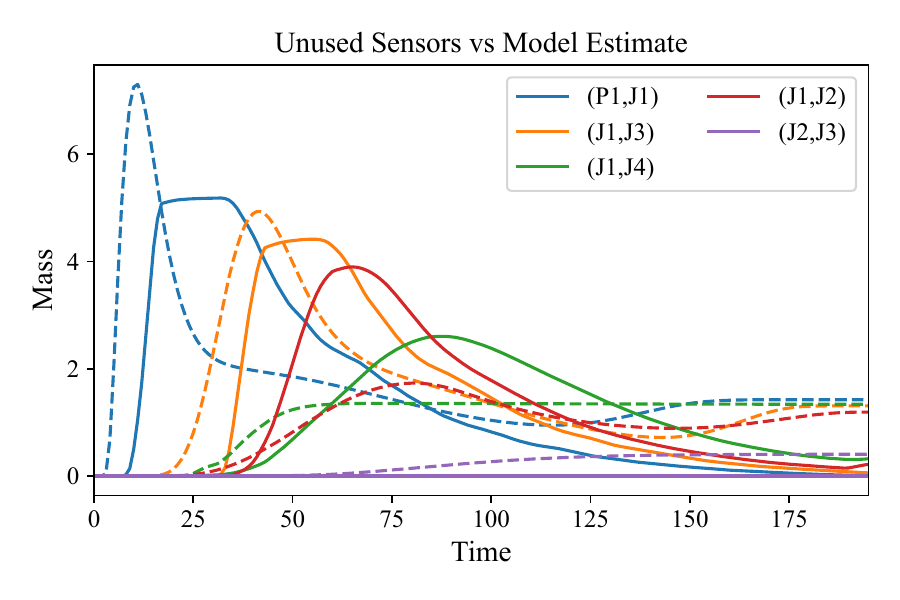}
    \end{minipage}\hfill
    \begin{minipage}{0.48\textwidth}
        \centering
        \includegraphics[width=\linewidth]{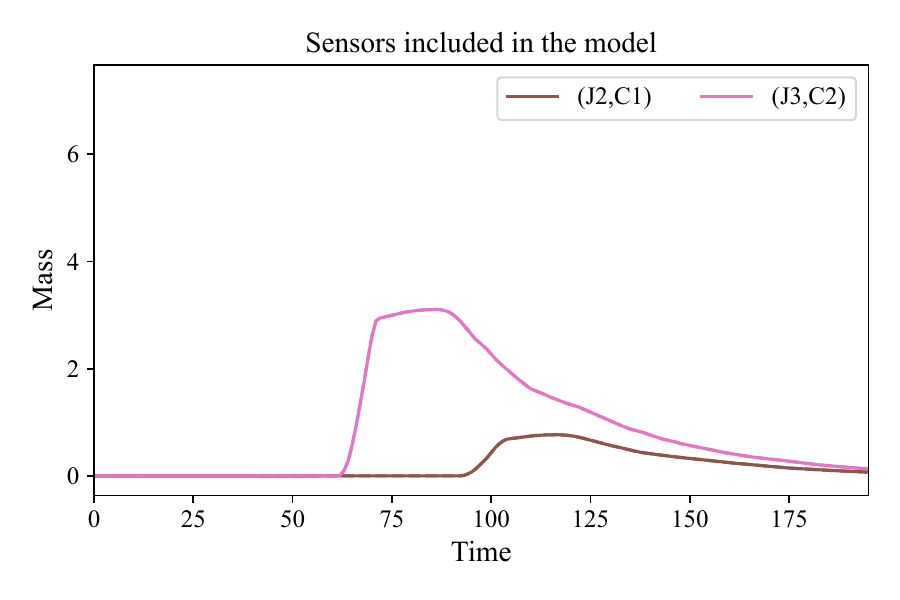}
    \end{minipage}
    \caption{Comparison between measured (solid) and reconstructed (dashed) pollutant mass signals for the \textit{contaminated tank experiment}.
Left: sensors that were not included in the reconstruction model.
Right: sensors used to estimate the pollutant distribution in the network. In this case, the reconstruction matches the sensor measurements exactly. }
    \label{fig:plots_side_by_side}
    \vspace{-8pt}
\end{figure*}

\begin{table}[h]
\centering
\begin{subtable}{0.49\linewidth}
\centering
\begin{tabular}{lcc}
\hline
  Pipe & Length & Diameter \\
    & [m] & [mm] \\
    \hline
    (J1,J2) & 20 & 13 \\
    (J1,J3) & 20 & 25 \\
    (J1,J4) & 20 & 25 \\
    (J2,C1) & 5 & 25 \\
    (J2,J3) & 20 & 25 \\
    (J3,C2) & 5 & 25 \\
    (J4,J2) & 20 & 20 \\
    (J4,J3) & 20 & 20 \\
    (P1,J1) & 20 & 25 \\
    (P2,J4) & 3 & 25 \\
    \hline
  \end{tabular}
\caption{\textit{Contaminated tank}}
\label{tab:contaminatedtank}
\end{subtable}
\begin{subtable}{0.49\linewidth}
\centering
\begin{tabular}{lcc}
\hline
  Pipe & Length & Diameter \\
    & [m] & [mm] \\
    \hline
    (J1,J2) & 20 & 25 \\
    (J1,J2) & 20 & 25 \\
    (J1,J4) & 20 & 13 \\
    (J2,J3) & 20 & 20 \\
    (J2,J4) & 20 & 20 \\
    (J3,C2) & 5 & 25 \\
    (J4,J3) & 20 & 25 \\
    (J4,J5) & 5 & 13 \\
    (J5,C1) & 5 & 25 \\
    (P1,J1) & 20 & 25 \\
    \hline
  \end{tabular}
\caption{ \textit{Contaminated pipe}}
\label{tab:contaminatedpipe}
\end{subtable}
\caption{Pipe properties for the two experiments.}
\label{tab:pipes-experiments}
\end{table}
\paragraph{Contaminated pipe}
This setup consists of one pumping station $P1$ and two consumer units ($C1$ and $C2$). The pipe properties are listed in \Cref{tab:contaminatedpipe}. Two sensors are included, located at the inlet of pipes $(J2,J3)$ and $(J4,J5)$. The network layout is shown in \Cref{fig:contaminatedpipe} and includes two identical parallel pipes between nodes $(J1, J2)$, although only one pipe is active at any time. The system has $n=56$ states and $k=2$ sensors. The time considered is $\T=300$~s.  Before the experiment begins, one of the two pipes is connected to a separate loop in which salty water circulates, making that pipe the only initially contaminated section of the network. The system is first operated with flow passing exclusively through the clean pipe. Then, at $t = 30$~s, the valves are switched so that the flow between $(J1, J2)$ is redirected through the contaminated pipe instead. 
\begin{figure}[ht]
    \centering
    \includegraphics[width=\linewidth]{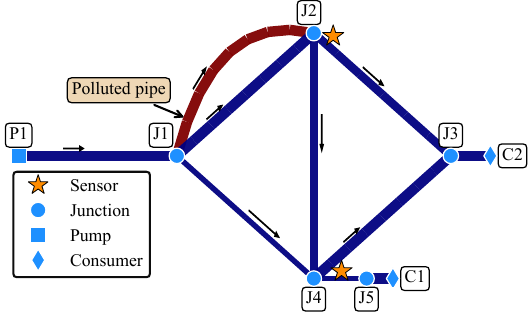}
    \caption{Water network for the \textit{contaminated pipe} scenario. Water flows according to the arrows, from the tanks to the consumers. The widths of the segments are scaled to the pipe diameters, but their lengths do not correspond to the actual pipe lengths.}
    \label{fig:contaminatedpipe}
\end{figure}
The recovered estimate is illustrated in \Cref{fig:simulationpipe} for a sequence of time steps. We compare it with data from the available sensors located in pipes downstream of the contamination (for upstream sensors, the estimate is correctly zero). This comparison and the readings from the sensors included in the model, is reported in \Cref{fig:plots_side_by_side_pipe}.
The method correctly recognizes that the contamination is coming from pipe $(J1,J2)$, but instead of a uniform spreading through the pipe (see \Cref{fig:contaminatedpipe}), suggests a concentration in the central segments (cf. \Cref{fig:simulationpipe} for $t=0$). That is also the reason why in \Cref{fig:plots_side_by_side_pipe}, the signal for $(J1,J2)$ is not matched by the estimate, as the sensor lies at the inlet of the pipe. 
The mass of salt in the system is approximately $40g$, and the model estimates $40.3 g$. The model also identifies the pipes in which contamination takes place, with a quantitative mismatch in $(J4,J3)$. The method indeed expects a uniform mixing in node $J4$, suggesting a signal closer to $(J4,J5)$ then what it actually takes place. 
\begin{figure*}[!t] 
	\centering
	\includegraphics[width=0.97\textwidth]{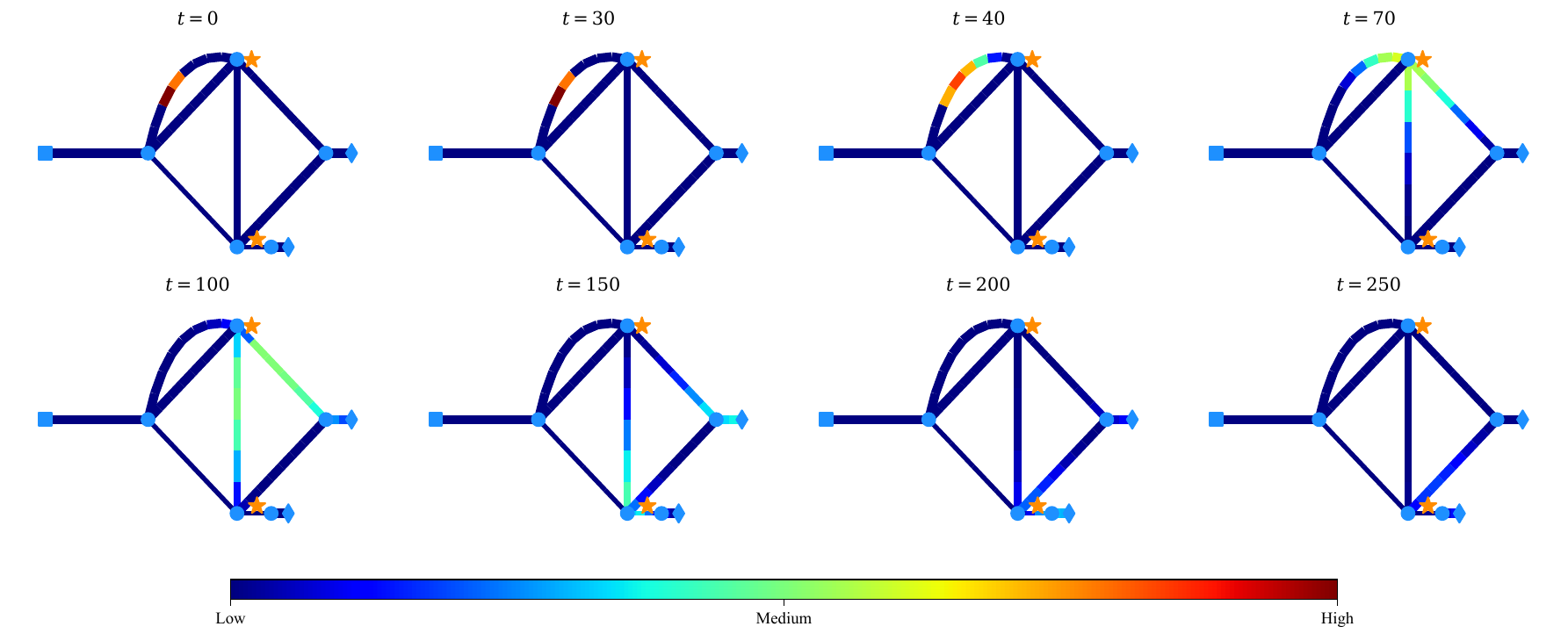}
	\caption{Concentration of salt over time in the reconstructed solution for the \textit{contaminated pipe} experiment, obtained using measurements over the full observation horizon ($\T=300$~s). Selected representative time instants are shown.}
	\label{fig:simulationpipe}
	\vspace{-8pt}
\end{figure*}
\begin{figure*}[!t]
    \centering
    \begin{minipage}{0.48\textwidth}
        \centering
        \includegraphics[width=\linewidth]{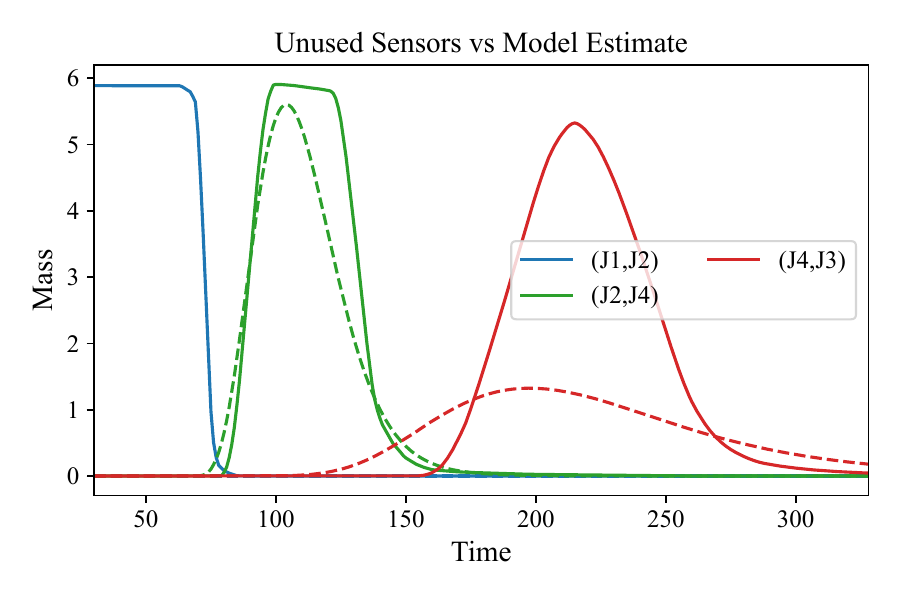}
    \end{minipage}\hfill
    \begin{minipage}{0.48\textwidth}
        \centering
        \includegraphics[width=\linewidth]{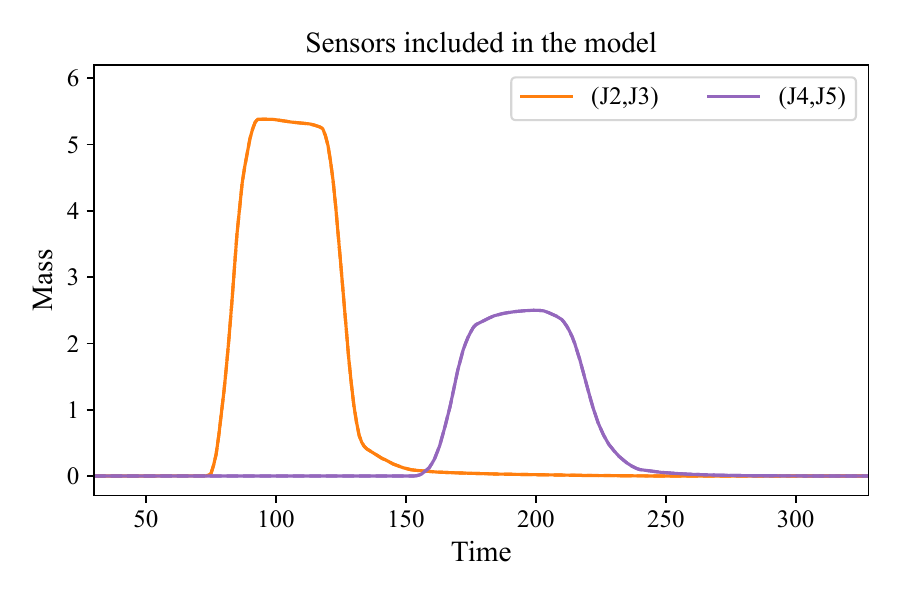}
    \end{minipage}
    \caption{Comparison between measured (solid) and reconstructed (dashed) pollutant mass signals for the \textit{contaminated pipe} experiment.
Left: sensors that were not included in the reconstruction model.
Right: sensors used to estimate the pollutant distribution in the network. In this case, the reconstruction matches the sensor measurements exactly. }
    \label{fig:plots_side_by_side_pipe}
    \vspace{-8pt}
\end{figure*}

\subsection{Discussion}
In both experiments, the methodology correctly identified the portion of the network from which the contamination had originated, reconstructing also with good accuracy the total amount of mass, and identifying the pipes subject to contamination. We observed, however, that in both experiments the method tends to initially localize the contamination within a small region and subsequently to diffuse it more than observed in the measured signals. We attribute this behavior to the probabilistic aspect introduced in \Cref{sec:waternetwork}, which could possibly be mitigated with a finer discretizations in time and space. 
On the other hand, while flow meters measure only mean velocity, real pipe flow is faster in the center and slower near the wall creating a diffusion \cite{pope2001turbulent}. Our probabilistic model has a type of stochastic variability that also results in diffusion. By contrast, EPANET’s  approach \cite{rossman2000epanet} tracks fixed “packages” of water that all move at the same speed along the pipe, and might underestimate how much the contaminant spreads inside the pipe.
Furthermore, while the model assumes perfect mixing at nodes, we observed in practice that when several inflows converge at a node, the division of water among outgoing pipes depends on the junction geometry and on the relative momentum of the incoming streams. 
We also observed that the estimates tend to anticipate the experimentally observed pollution, which may be partly explained by measurement errors. Although the probabilistic framework is robust to moderate perturbations in flow values, this robustness does not extend to errors in the sparsity pattern induced in the transition matrices. Since flow meters may deviate significantly during transient or low-flow conditions, and conductivity sensors may exhibit temporal inertia, mismatches between transport and observation data may arise. As a result, the physically expected evolution may become infeasible for the reconstruction problem when multiple sensors are imposed simultaneously, or otherwise require artificial mass corrections to restore balance.

\section{Conclusion}
\label{sec:conclusion}

In this paper we consider a version of the Schr{\"o}dinger bridge problem with partially observed marginals. We developed a theoretical framework including duality results and observability results along with a scalable method to compute optimal solutions.

The method was validated on experimental data from a laboratory-scale water distribution network, successfully reconstructing contaminant transport and identifying the source from sparse sensor measurements. Despite some model–experiment discrepancies, the results demonstrate the effectiveness and robustness of the proposed approach. A future direction is to study the sensitivity of the method to  model and sensor errors.
This also relates to the sensor placement in the network, and if good designs can be obtained by studying the observability condition.  
Further work also includes generalizing the approach to unbalanced problems where  mass is continuously created or destroyed in the network.

\section*{Declaration of Generative AI and AI-assisted technologies in the writing process}

During the preparation of this work, the authors used ChatGPT (OpenAI) to assist with text editing and to improve clarity. After using this tool, the authors reviewed and edited the content as needed and take full responsibility for the content of the publication.

\appendices 

\section{Proof of \Cref{prop:existence}}
\label{app:exist}
\begin{proof}
    First, observe that the feasible set of \eqref{eq:primal} is closed, as it is defined by linear constraints, and nonempty by the feasiblity assumption. The objective function can be written as a sum of terms
      \[
F_t(M_t) := \ccD\big(M_t \,\big|\, \diag(M_t\one)A_t\big),
\]
    which are non-negative, convex and lower semi-continuous. The infimum is thus finite. To prove existence of a minimizer it remains to analyze unbounded feasible directions along which the objective does not increase, which could prevent attainment. We characterize directions $Z_{[0:\T-1]}$  such that, for any feasible $M_{[0: \T-1]}$ and any scalar $\alpha \geq 0$, the perturbation $M_t+\alpha Z_t$ is feasible for all $t$. By linearity of the constraints \eqref{eq:obst}--\eqref{eq:matching}, this is equivalent to $Z_t \geq 0$ and
\begin{subequations}
\label{eq:dir_constraints}
\begin{align}
& C Z_t \one = 0, 
\qquad \quad \ \text{for } t=0,\dots,\T-1, 
\label{eq:dir_obst}\\
& C Z_{\T-1}^\top \one = 0, 
\label{eq:dir_final}\\
& Z_t \one = Z_{t-1}^\top \one, 
\qquad \text{for } t=1,\dots,\T-1.
\label{eq:dir_matching}
\end{align}
\end{subequations}
Among feasible directions \(Z_t\), we evaluate the recession rate (asymptotic growth) of each term \(F_t\) along the ray \(M_t+\alpha Z_t\):
\[
F_t^\infty(Z_t):=\lim_{\alpha\to\infty}\frac{F_t(M_t+\alpha Z_t)-F_t(M_t)}{\alpha}.
\]
A direct computation gives
\[
F_t^\infty(Z_t)=\ccD \big(Z_t\,\big|\,\diag(Z_t\one)A_t\big)\ge 0,
\]
with equality if and only if
\begin{equation}
\label{eq:scaling}
Z_t=\diag(Z_t\one)\,A_t.
\end{equation}
Thus the only feasible recession directions with zero recession value satisfy \eqref{eq:dir_constraints} and \eqref{eq:scaling}.  Combining these conditions, with $z_t := Z_t\one$, we obtain
\[
z_{t+1} = A_t^\top z_t, \qquad t=0,\dots,\T-1.
\]
Define $\Phi_0 := I$ and $\Phi_t := A_{t-1}^\top \cdots A_0^\top$ for $t\ge 1$.
Then $z_t = \Phi_t z_0$ for $t=0,\dots,\T$.
Hence the constraints \eqref{eq:dir_obst}--\eqref{eq:dir_final} are equivalent to
\[
z_0 \in \ker(\mathcal O_\T), 
\qquad
\mathcal O_\T :=
\begin{bmatrix}
C\Phi_0 \\
\vdots \\
C\Phi_\T
\end{bmatrix}.
\]
Furthermore, if we consider $z_0 \in \ker(\mathcal{O}_\T) \cap \mR^n_+$, then
\[
0=C \Phi_tz_0 = \sum_i (z_0)_i  C \Phi_t e_i
\] 
is a combination of non-negative vectors, and  $(z_0)_i >0 \implies C \Phi_t e_i=0 \ \forall t$, meaning that the $i$-th column of $\mathcal{O}_\T$ contains only zeros. We define the index sets
\[
\mathcal I := \big\{\, i\in\{1,\dots,n\} : \mathcal O_\T e_i = 0 \,\big\},
\quad
\mathcal J := \{1,\ldots,n\} \setminus \mathcal{I} .
\]
We showed that if $z_0\in\ker(\mathcal O_\T)\cap\mR^n_+$, then $\supp(z_0) \subseteq \mathcal I$.
By definition of $\mathcal I$ and nonnegativity of the matrices $A_t$, the set $\mathcal I$ is forward invariant with respect to the supports of $A_t$, in the sense that if $i\in\mathcal I$ and $(A_t)_{ij}>0$, then $j\in\mathcal I$.  Since \(\supp(M_t)\subseteq\supp(A_t)\), no feasible \(M_t\) can move mass from \(\mathcal I\) to \(\mathcal J\).
Let $M_{[0:\T-1]}$ be any feasible solution. We construct a feasible $\tilde M_{[0:\T-1]}$ with no larger objective value by keeping all rows indexed by $\mathcal J$ unchanged and replacing the rows indexed by $\mathcal I$ by KL-minimizing rows with the same row sums. In particular, 
\[
(\tilde M_t)_{ij}=
\begin{cases}
(M_t)_{ij}, 
& \text{if } i\in\mathcal J,\ j\in\{1,\dots,n\},\\[1mm]
0, 
& \text{if } i\in\mathcal I,\ j\in\mathcal J,\\[1mm]
(\tilde M_t\one)_i\,(A_t)_{ij},
& \text{if } i\in\mathcal I,\ j\in\mathcal I,
\end{cases}
\]
where the row sums \((\tilde M_t\one)_i\) for \(i\in\mathcal I\) are chosen recursively so that the matching constraints \(\tilde M_t\one=\tilde M_{t-1}^\top\one\) hold. 
This is feasible because the rows in $\mathcal J$ are not changed, no mass can leave $\mathcal I$ toward $\mathcal J$, and the matching constraints only prescribe row/column sums. 
Moreover, the observation constraints are unchanged: by definition of $\mathcal I$, no state in $\mathcal I$ contributes to any measured component at any time, so modifying only rows with index $i\in\mathcal I$ does not affect \eqref{eq:obst}--\eqref{eq:finalT}.
Due to the properties of KL divergence, the objective function decouples row-wise, and it is nonnegative, vanishing if and only if
$(M_t)_{ij}=(M_t\one)_i (A_t)_{ij} \ \forall \, j$.
Hence, for each $t$, the above modification replaces every row indexed by $\mathcal I$ by its unique minimizer given its row sum, yielding, for $F=\sum_t F_t,$
\[
F(\tilde M)\le F(M),
\]
with equality if and only if \((M_t)_{ij}=(M_t\one)_i(A_t)_{ij}\) for all \(t\), all \(i\in\mathcal I\), and all \(j\in\{1,\dots,n\}\).
Therefore, the infimum of \eqref{eq:primal} is attained over the restricted feasible set
\[
\mathcal F^\star  \!:= \! \Big\{M \text{ feasible}  \!: \!
(M_t)_{ij}=(M_t\one)_i (A_t)_{ij} \ \forall\, t, \forall\, i \! \in \! \mathcal{I},\ \forall j
\Big\} \!.
\]
On $\mathcal F^\star$, any recession direction affects only the rows indexed by $\mathcal I$ and does not change the value of the objective. Therefore, along any unbounded feasible direction the objective remains constant.
Since $\mathcal F^\star$ is nonempty and closed, the objective attains its minimum on $\mathcal F^\star$ by \cite[Thm.~27.3]{Rockafellar1970}, and a minimzer exists on $\mathcal F$.
\end{proof}

\section{Proof of \Cref{prop:uniqueness}}
\label{app:proof1}
We show that systems \eqref{eq:systemA} and \eqref{eq:systemM} have the same observability properties. Introduce the auxiliary system
\begin{equation}
\label{eq:nonsystem}
\begin{aligned}
\mu_{t+1} &= \overline A_t^\top \mu_t,\\
\rho_t &= C\mu_t,
\end{aligned}
\end{equation}
where $
\overline A_t^\top := A_t^\top \overline C^\top \overline C .
$
Let $
P:=\overline C^\top \overline C$ and $Q:=C^\top C$.
The matrix \(Q\) extracts the coordinates corresponding to the observed states, while \(P\) extracts the complementary coordinates. In particular $CQ=C$ and $CP=0$.

Denote by \(\mathcal O_r\), \(\tilde{\mathcal O}_r\), and \(\overline{\mathcal O}_r\) the \(r\)-step observability matrices associated with the pairs
\[
(A_t^\top,C), \qquad (\mathcal A_t^\top,C), \qquad (\overline A_t^\top,C),
\]
respectively. We will prove by induction on \(r\) that
\begin{equation}
\label{eq:kernels-equal}
\ker(\mathcal O_r)=\ker(\tilde{\mathcal O}_r)=\ker(\overline{\mathcal O}_r),
\qquad \forall r\ge 1.
\end{equation}
For \(r=1\), the three observability matrices reduce to \(C\), so the claim is immediate.
Assume now that \eqref{eq:kernels-equal} holds for some \(r-1\ge 1\). We prove it for \(r\).

First, we show \(\ker(\mathcal O_r)=\ker(\overline{\mathcal O}_r)\).
The first \(r-1\) block rows of \(\mathcal O_r\) and \(\overline{\mathcal O}_r\) have the same kernel by the inductive hypothesis.  Let \(x\in\ker(\mathcal O_r)\). Then, consider the \(r\)-th observability equation
\[
C A_{r-2}^\top A_{r-3}^\top \cdots A_0^\top x = 0.
\]
Insert \(I=P+Q\) after each factor $A_t$:
\begin{align*}
0
&= C A_{r-2}^\top (P+Q) A_{r-3}^\top (P+Q)\cdots A_0^\top (P+Q)x .
\end{align*}
Expanding the product, every term containing at least one factor \(Q\) vanishes. Indeed, consider such a term and let \(Q\) be its rightmost occurrence (i.e., the one closest to \(x\)). Then all factors to its right are equal to \(P\), so this term reduces to one of the observability conditions of order at most \(r-1\). Hence it is zero by the inductive hypothesis \(\ker(\mathcal O_{r-1})=\ker(\overline{\mathcal O}_{r-1})\). Therefore, the only surviving term is the one containing only \(P\)'s, namely
\[
C A_{r-2}^\top P A_{r-3}^\top P \cdots A_0^\top P x = 0.
\]
This is exactly the \(r\)-th observability equation for \(\overline{\mathcal O}_r\). Thus \( x\in\ker(\mathcal O_r)\ \Longrightarrow\ x\in\ker(\overline{\mathcal O}_r) \), and because the argument relies on a reversible chain of equalities, the converse also holds.

Now we show  \(\ker(\tilde{\mathcal O}_r)=\ker(\overline{\mathcal O}_r)\).
Recall that the optimal dynamics \(\mathcal A_t\) has the form
\[
\mathcal A_t^\top=\diag(w_{t+1})\,A_t^\top\,\diag(u_t\oslash w_t).
\]
In the product defining the \(r\)-th observability equation, the factors \(\diag(w_t)^{-1}\) and \(\diag(w_t)\) cancel telescopically (with \(\diag(w_0)=I\)), so only a left factor \(C\diag(w_{r-1})\) remains. 
This left multiplication can be removed without affecting the kernel, because
\[
C\diag(w_{r-1}) = \big(C\diag(w_{r-1})C^\top\big)\,C,
\]
and \(C\diag(w_{r-1})C^\top\in\mathbb R_{\geq 0}^{k\times k}\) is diagonal and invertible, since all entries of \(w_{r-1}\) are positive.
Moreover, since \(u_t=\exp(C^\top\lambda_t)\), the diagonal matrix
\[
D_t:=\diag(u_t)=\diag(\exp(C^\top\lambda_t))
\]
satisfies $D_t = D_t Q + P$, 
because \(u_t\equiv 1\) on the unobserved coordinates selected by \(P\). 
Therefore, the \(r\)-th observability equation for \(\tilde{\mathcal O}_r\) is equivalent to
\[
C A_{r-2}^\top D_{r-2} \cdots A_0^\top D_0 x = 0.
\]

Let \(x\in\ker(\tilde{\mathcal O}_r)\). By the inductive hypothesis, the first \(r-1\) observability equations for \(\tilde{\mathcal O}_{r-1}\) and \(\overline{\mathcal O}_{r-1}\) are equivalent. To compare the \(r\)-th equation, expand
\[
C A_{r-2}^\top (D_{r-2}Q+P)\cdots A_0^\top (D_0Q+P)x.
\]
As before, every term containing at least one factor \(Q\) vanishes by the first \(r-1\) observability equations and the inductive hypothesis, and the only surviving term is again the \(r\)-th observability equation for \(\overline{\mathcal O}_r\).
Hence,
\[
x\in\ker(\tilde{\mathcal O}_r)\ \Longrightarrow\ x\in\ker(\overline{\mathcal O}_r).
\]
The converse inclusion follows from the same expansion argument, obtaining \(
\ker(\tilde{\mathcal O}_r)=\ker(\overline{\mathcal O}_r).
\)

In conclusion, \(\ker(\mathcal O_r)=\ker(\tilde{\mathcal O}_r)\) for all \(r\), and in particular for \(r=\T+1\).

\bibliographystyle{abbrv}

\bibliography{./references}

\newpage


\end{document}